\documentclass[12pt]{article}
\usepackage{verbatim}
\usepackage{fancyhdr}
\usepackage{graphicx}
\usepackage{subfigure}
\usepackage{amsmath}
\usepackage{amssymb}
\usepackage{verbatim}
\usepackage{fancyhdr}
\usepackage{indentfirst}
\usepackage{booktabs}
\usepackage{amssymb}
\usepackage{amsmath,amssymb,graphicx}
\usepackage{amsthm}
\usepackage{subfigure}
\usepackage{color}
\usepackage{mathrsfs}
\usepackage{multirow}
\usepackage{listings}
\usepackage{array}
\usepackage{booktabs}
\usepackage{titlesec}
\usepackage{epsfig}
\usepackage{float}
\usepackage{appendix}

\newtheorem{thm}{Theorem}[section]
\newtheorem{ex}{Example}[section]
\newtheorem{remark}{Remark}[section]
\newtheorem{lemma}{Lemma}[section]

\allowdisplaybreaks[4]
\newcommand{\jr}{{j+\frac{1}{2}}}
\newcommand{\jl}{{j-\frac{1}{2}}}

\newcommand{\jlr}{{j \mp \frac{1}{2}}}
\newcommand{\ir}{{i+\frac{1}{2}}}
\newcommand{\il}{{i-\frac{1}{2}}}

\providecommand{\mean}[1]{ \{\mspace{-6.0mu}\{ #1 \}\mspace{-6.0mu}\} }
\providecommand{\jump}[1]{ [\mspace{-2.5mu}[ #1 ]\mspace{-2.5mu}] }

\title{An ultraweak-local discontinuous Galerkin method
for PDEs with high order spatial derivatives}
\author{Qi Tao\footnote{School of Mathematical Sciences, University
of Science and Technology of China, Hefei, Anhui 230026, P.R. China.
Email: taoq@mail.ustc.edu.cn. Research supported by China Scholarship
Council.},~~
Yan Xu\footnote{Corresponding author. School of Mathematical Sciences, University of Science
and Technology of China, Hefei, Anhui 230026, P.R. China.
Email: yxu@ustc.edu.cn. Research supported by National Numerical Windtunnel grants NNW2019ZT4-B08, Science Challenge Project TZZT2019-A2.3, NSFC grants 11722112.},
~~ Chi-Wang Shu\footnote{Division of Applied Mathematics, Brown University,
Providence, RI 02912, USA.
Email: shu@dam.brown.edu. Research supported by NSF grant DMS-1719410.}
}
\date{}
\begin{document}
\numberwithin{equation}{section}
\maketitle
\textbf{Abstract:}
In this paper, we develop a new discontinuous Galerkin method for
solving several types of partial differential equations (PDEs) with
high order spatial derivatives. We combine the advantages of local
discontinuous Galerkin (LDG) method and ultra-weak discontinuous
Galerkin (UWDG) method. Firstly, we rewrite the PDEs with high
order spatial derivatives into a lower order system, then apply
the UWDG method to the system.
We first consider the fourth order and fifth order nonlinear PDEs
in one space dimension,
and then extend our method to general high order problems and two
space dimensions.
The main advantage of our method over the LDG method is that we
have introduced fewer auxiliary variables, thereby reducing memory
and computational costs.  The main advantage of our method over
the UWDG method is that no internal penalty terms are necessary
in order to ensure stability for both even and odd order PDEs.
We prove stability of our method in the general
nonlinear case and provide optimal error estimates for linear PDEs
for the solution itself as well as for the auxiliary variables
approximating its derivatives.  A key ingredient in the proof
of the error estimates is the construction of the relationship
between the derivative and the element
interface jump of the numerical solution and
the auxiliary variable solution of the solution derivative.
With this relationship, we can then use the discrete Sobolev and
Poincar\'{e} inequalities to obtain the optimal error estimates.
The theoretical findings are confirmed by numerical experiments. 

\section{Introduction}

In this paper, we propose a new class of discontinuous Galerkin (DG)
methods for solving several types of partial differential equations (PDEs)
with high order spatial derivatives. The first two examples we
consider are:
\begin{itemize}
  \item  The fourth order equation
         \begin{align}\label{fourth_equation}
         u_{t}+(b(u)u_{xx})_{xx}=0, \quad b(u)\geq 0
         \end{align}
  \item  The fifth order equation
         \begin{align}\label{fifth_equation}
         u_{t}+f(u_{xx})_{xxx}  =0.
         \end{align}
\end{itemize}
The boundary conditions are assumed to be periodic for simplicity,
although most of our discussions can be adapted for other types
of boundary conditions. These equations are classical model equations
for many very important physical applications. The fourth order
problem has wide applications in the modeling of thin beams and plates,
strain gradient elasticity, and phase separation in binary mixtures
\cite{fourth}.  The fifth order nonlinear
evolution equation is known as the critical surface-tension model
\cite{fifth}.

Discontinuous Galerkin (DG) methods are a class of finite element
methods (FEMs) using completely discontinuous basis functions. The
first DG method was introduced in 1973 by Reed and Hill
\cite{reed1973triangular} in the framework of neutron transport.
It was later developed for time-dependent nonlinear hyperbolic
conservation laws, coupled with the Runge-Kutta time discretization,
by Cockburn et al \cite{cockburndiscontinuous1990,
cockburndiscontinuous1989, cockburndiscontinuous1998, ShuDG}. Since then,
the DG method has been intensively studied and successfully applied
to various problems in a wide range of applications due to its
flexibility with meshing, its compactness and its high parallel
efficiency. For the equations containing higher order spatial derivatives,
there are several different ways to approximate them by
discontinuous Galerkin methods.  One way is to use the local discontinuous
Galerkin (LDG) method \cite{cockburn1998local, dong, Ji, LiuLOCAL, XSrev, Yankdv,
Yanlocal}.  The idea of the LDG methods is to rewrite the
equations with higher order spatial derivatives into a first order
system, then apply the DG method to this system and design suitable numerical
fluxes to ensure stability.  Another way is to use the penalty methods
that add penalty terms at cell interfaces in the DG formulation
for numerical stability \cite{IP1,IP2}.  The third way is to use the ultra-weak DG
(UWDG) methods \cite{chengultraweak}.  It is based on
repeated integration by parts to move all spatial derivatives to
the test function in the weak formulation, and on a careful choice of
the numerical fluxes to ensure stability and optimal accuracy.
Unlike the traditional LDG method, the UWDG method can be applied
without introducing any auxiliary variables
or rewriting the original equation into a system.  Recently, Liu et al.
introduced a mixed DG method \cite{Liumix}, by first rewriting the
fourth order PDEs into a second order coupled system and then using
a direct DG discretization for the second order system.
$L^{2}$ stability was obtained without internal penalty.

In this paper, we design a new class of DG methods, combining the
advantages of LDG and UWDG methodologies, to solve PDEs
with high order spatial derivatives.  The two PDEs
(\ref{fourth_equation}) and (\ref{fifth_equation}) are used
first as examples to develop our method.  The method is then
extended to a wider class of PDEs both in one and in two
dimensions.  Similar to the mixed DG method in \cite{Liumix}, we
first rewrite the higher order equation into a lower order
(but not all first order) system.  For example, we rewrite the
fourth order problem into a second order system and rewrite the
fifth order problem into a system with two second order equations
and a first order equation, then we repeat the application of
integration parts, and choose suitable numerical fluxes to ensure
stability.  For the equations with spatial derivative order less than
or equal to three, our method will be the same as the LDG methods
or ultra-weak DG method, but for higher order PDEs our method
combines the advantages of the two type of methods, and is more
efficient. It is known that the proof of optimal accuracy
for LDG methods solving high order time-dependent wave equations is
very difficult.  The work in \cite{Xuoptimal} by Xu and Shu might
be the first to prove optimal order of accuracy in $L^{2}$ for
not only the solution but also the auxiliary variables.
In their work, the main idea is to derive energy stability for
the auxiliary variables in the
LDG scheme by using the scheme and its time derivatives.
In \cite{fu} Fu et al. identified a sub-family of the numerical
fluxes by choosing the coefficients in the linear combinations, so
that the solution and some auxiliary variables of the proposed DG
methods are optimally accurate in the $L^{2}$ norm.
In \cite{dong} Dong and Shu proved the optimal error estimates for the
higher even-order equations,
including the cases both in one dimension and in
multidimensional triangular meshes.  In this paper,
we prove the optimal error estimates for both the even order equations
and the odd order equations. The main idea is to use an important
relationship between the derivative and the element interface jump of
the numerical solution and the auxiliary variable numerical solution of
the derivative \cite{Wangimex2015, Wangimex2016}. Then we can obtain
suitable estimates to the auxiliary variables, which lead to
the optimal error estimates for both the numerical solution and
the auxiliary variables.  This is a different approach from
that in \cite{dong,Xuoptimal}, since in this way we do not
need to estimate many energy equations, and can get the relationship
between the solution and auxiliary variables directly.

The organization of the paper is as follows. In Section 2, we introduce
some notations and projections that will be used later. In Section 3,
the scheme for the fourth order equation is discussed, including
the discussion on the $L^{2}$ stability and optimal error estimates. In
Section 4, we follow the lines of Section 3 and consider the fifth
order equation. In Section 5, we extend the schemes in Sections 3 and 4
to arbitrary even and odd order equations, respectively.
We also extend the scheme for the fourth order equations to
multidimensional Cartesian meshes as an example of multi-dimensions in Section 6.
The theoretical results are confirmed numerically in Section 7. In
Section 8, we give some concluding remarks.

\section{Notations and projections}

In this section, we will introduce some notations, definitions and
projections that will be used later for the one-dimensional equations.

Throughout this paper, we adopt standard notations for the Sobolev
spaces such as $W^{m,p}(D)$ on the subdomain $D\in \Omega$ equipped
with the norm $\|\cdot\|_{m,p,D}$.  If $D=\Omega$, we omit the
index $D$; and if $p=2$, we set $W^{m,p}(D)=H^{m}(D)$,
$\|\cdot\|_{m,p,D}=\|\cdot\|_{m,D}$; and we use $\|\cdot\|_{D}$ to
denote the $L^{2}$ norm in $D$.

\subsection{Basic notations}

Let $\Omega=[0, 2\pi]$ and $0=x_{\frac{1}{2}}<x_{\frac{3}{2}}<\cdots
<x_{N+\frac{1}{2}}=2\pi$ be $N+1$ distinct points on $\Omega$. For
each positive integer $r$, we define $Z_{r}=(1,2,\cdots,r)$ and
denote by
$$
I_{j}=(x_{\jl},x_{\jr}),~~~x_{j}=\frac{1}{2}(x_{\jl}+x_{\jr}), ~~j\in Z_{N},
$$
the cells and cell centers, respectively. Let $h_{j}=x_{\jr}-x_{\jl}$,
and $h=\max\limits_{j} h_{j}$. We assume that the mesh is regular. Define
$$V_{h}=\{v_{h}:v_{h}|_{I_{j}}\in \mathcal{P}^{k}(I_{j}), j\in Z_{N}\}$$
to be the finite element space, where $\mathcal{P}^{k}$ denotes the space
of polynomials of degree at most $k$. For any $v\in V_{h}$,
$v_{\jr}^{+}$ and $v_{\jr}^{-}$ denote the right and left limit values
of $v$ at $\jr$, respectively. As usual, the average and the jump of
the function $v$ at $\jr$ are denoted as
$$\mean{v}_{\jr}=\frac{1}{2}
(v_{\jr}^{+}+v_{\jr}^{-}), \quad \jump{v}_{\jr}=v_{\jr}^{+}-v_{\jr}^{-},$$
respectively.

\subsection{Projections}

Next, we will introduce some projections used in the error estimates.
For example, we can choose the Gauss-Radau projections
$P_{h}^{\pm}$ into $V_{h}$, such that for any $u$ we have:
\begin{align}\label{guass_radau_projection}
\int_{I_{j}}uv_{h}dx=\int_{I_{j}}P_{h}^{\pm}uv_{h}dx,\quad
P_{h}^{\pm}u\left(x_{\jlr}^{\pm}\right)=u\left(x_{\jlr}\right),
\end{align}
$\forall j\in Z_{N}, v_{h}\in \mathcal{P}^{k-1}(I_{j})$.
Furthermore, for $k\geq1$ we can define the projection $P_{1h}^{\pm}$
into $V_{h}$ such that, for any $u$, the projection $P_{1h}^{\pm} u$
satisfies: $ \forall j\in Z_{N}$
\begin{align}\label{projection1.1}
\int_{I_{j}}uv_{h}dx=\int_{I_{j}}P_{1h}^{\pm}uv_{h}dx,
\end{align}
for any $v_{h}\in \mathcal{P}^{k-2}(I_{j})$ and
\begin{align}\label{projection1.2}
P_{1h}^{\pm}u\left(x_{\jlr}^{\pm}\right)=u\left(x_{\jlr}\right),
\quad (P_{1h}^{\pm}u)_{x}\left(x_{\jlr}^{\pm}\right)=u_{x}\left(x_{\jlr}
\right).
\end{align}
Similarly, for $k\geq2$ we can define the projection $P_{2h}^{\pm}$
into $V_{h}$ such that, for any $u$, it satisfies:
\begin{align}\label{projection2.1}
\int_{I{j}}uv_{h}dx=\int_{I{j}}P_{2h}^{\pm}uv_{h}dx,
\end{align}
and
\begin{align}\label{projection2.2}
P_{2h}^{\pm}u\left(x_{\jlr}^{\pm}\right)=u\left(x_{\jlr}\right),
~(P_{2h}^{\pm}u)_{x}\left(x_{\jlr}^{\pm}\right)=u_{x}\left(x_{\jlr}\right),
~(P_{2h}^{\pm}u)_{xx}\left(x_{\jlr}^{\pm}\right)=u_{xx}\left(x_{\jlr}\right),
\end{align}
for any $j\in Z_{N}$, $v_{h}\in \mathcal{P}^{k-3}(I_{j})$.
We will use different projections according to the need in each proof.
For all these projections, the following inequality holds
\cite{ciarlet2002finite}:
\begin{align}
\|u^{e}\|+h\|u^{e}\|_{\infty}+h^{\frac{1}{2}}\|u^{e}\|_{\Gamma_{h}}\leq Ch^{k+1}\|u\|_{k+1},
\end{align}
where $u^{e}=\pi_{h}^{\pm}u-u$, $\pi_{h}=P_{h},~P_{1h},~P_{2h}$, and $\Gamma_{h}$ denotes the set of boundary points of all elements $I_{j}$, and $C$ is a positive constant dependent on $k$ but not on $h$.

\section{The fourth order problem}

We start from the fourth order problem. Firstly, we consider the following one-dimensional nonlinear equation
\begin{align}\label{linear_forth}
&u_{t}+(b(u)u_{xx})_{xx}=0, \quad b(u)\geq 0,~~~~(x,t)\in [0,2\pi]\times (0,T],\\
&u(x,0)=u_{0}(x),~~~~ x\in \mathbb{R},
\end{align}
where $u_{0}(x)$ is a smooth function. Without loss of generality, we only consider the periodic boundary conditions.

\subsection{The numerical scheme}

Before we introduce our DG method, we rewrite the fourth order
equation (\ref{linear_forth}) into a system of second order equations
\begin{align}\label{linear_forth_system}
&u_{t}+v_{xx}=0,\\
&v-b(u)w=0,\\
&w-u_{xx}=0.
\end{align}
Notice that, unlike the LDG method, we stop at second
order equations and do not go all the way to a first order
system.
Our DG method is defined as follows: find
$u_{h},~v_{h},~w_{h}\in V_{h}$ such that for all $p,~s,~q\in V_{h}$, we have
\begin{align}
&((u_{h})_{t},p)_{j}+(v_{h},p_{xx})_{j}+\widetilde{v_{x}}p^{-}|_{\jr}-\widetilde{v_{x}}p^{+}|_{\jl}
-\widehat{v}p_{x}^{-}|_{\jr}+\widehat{v}p_{x}^{+}|_{\jl}=0,\label{linear_forth_scheme1}\\
&(v_{h},s)_{j}-(b(u_{h})w_{h},s)_{j}=0,\label{linear_forth_scheme3}\\
&(w_{h},q)_{j}-(u_{h},q_{xx})_{j}-\widetilde{u_{x}}q^{-}|_{\jr}+\widetilde{u_{x}}q^{+}|_{\jl}
+\widehat{u}q_{x}^{-}|_{\jr}-\widehat{u}q_{x}^{+}|_{\jl}=0.\label{linear_forth_scheme2}
\end{align}
Here $\displaystyle (u, v)_{j}=\int_{I_{j}}uvdx$ and
$\widehat{v}$, $\widetilde{v_{x}}$, $\widehat{u}$,
$\widetilde{u_{x}}$ are the numerical fluxes. The terms involving
these fluxes appear from repeated integration by parts, and a
suitable choice for these fluxes is the key ingredient for the stability
of the DG scheme. We can take either of the following four choices of
alternating fluxes for these four fluxes
\begin{align}\label{flux_linear_forth1}
\widehat{v}=v^{-}_h,~ \widetilde{v_{x}}=(v_{h})_{x}^{-},~
\widehat{u}=u_{h}^{+},~ \widetilde{u_{x}}=(u_{h})_{x}^{+};
\end{align}
\begin{align}\label{flux_linear_forth2}
\widehat{v}=v^{+}_h,~ \widetilde{v_{x}}=(v_{h})_{x}^{+},~
\widehat{u}=u_{h}^{-},~ \widetilde{u_{x}}=(u_{h})_{x}^{-};
\end{align}
\begin{align}\label{flux_linear_forth3}
\widehat{v}=v^{-}_h,~ \widetilde{v_{x}}=(v_{h})_{x}^{+},~
\widehat{u}=u_{h}^{-},~ \widetilde{u_{x}}=(u_{h})_{x}^{+};
\end{align}
\begin{align}\label{flux_linear_forth4}
\widehat{v}=v^{+}_h,~ \widetilde{v_{x}}=(v_{h})_{x}^{-},~
\widehat{u}=u_{h}^{+},~ \widetilde{u_{x}}=(u_{h})_{x}^{-}.
\end{align}
It is crucial that $\widehat{v}$ and $\widetilde{u_{x}}$ come
from the opposite sides, and $\widetilde{v_{x}}$ and
$\widehat{u}$ come from the opposite sides (alternating
fluxes).
\begin{remark}
For the numerical fluxes,  we can also  take the following numerical fluxes
\begin{subequations}\label{flux4}
\begin{align}\label{flux4_v}
&\widehat{v}=\theta v^-_{h} +(1-\theta)v^+_h,\quad \widetilde{v_{x}}=\theta (v_h)^-_x +(1-\theta )(v_h)^+_x,\\
&\widehat{u}=\theta u^+_{h} +(1-\theta)u^-_h,\quad \widetilde{u_{x}}=\theta (u_h)^+_x +(1-\theta)(u_h)^-_x,\label{flux4_u}
\end{align}
\end{subequations}
where $0\leq\theta\leq 1$.
For $\theta= 1/2$,  we would have the central fluxes as in \cite{Liumix} for the linear case.
We note that, unlike in the UWDG method \cite{chengultraweak}, here we
do not need to add extra internal penalty terms to ensure stability.
\end{remark}

\subsection{Stability analysis}

In this subsection, we will show the stability property of the
scheme (\ref{linear_forth_scheme1})-(\ref{linear_forth_scheme2}) with
the choice of fluxes (\ref{flux_linear_forth1})-(\ref{flux4}).
\begin{thm}\label{stability_linear_forth}
Our numerical scheme
(\ref{linear_forth_scheme1})-(\ref{linear_forth_scheme2}) with the
choice of fluxes (\ref{flux_linear_forth1})-(\ref{flux4})
is $L^{2}$ stable, i.e.
\begin{align}\label{stability_equation1}
\frac{1}{2}\frac{d}{dt}\int_{\Omega}u_{h}^{2}(x,t)dx+\int_{\Omega}b(u_{h})w_{h}^{2}(x,t)dx=0.
\end{align}
\end{thm}
\begin{proof}
We integrate by parts in the scheme (\ref{linear_forth_scheme1}) and
(\ref{linear_forth_scheme2}) and sum over $j$ to obtain
\begin{align}
\label{shuadd1}
&((u_{h})_{t},p)_{\Omega}-((v_{h})_{x},p_{x})_{\Omega}+B_{1}(v_{h},p)=0,\\
\label{shuadd2}
&(v_{h},s)_{\Omega}-(b(u_{h})w_{h},s)_{\Omega}=0,\\
\label{shuadd3}
&(w_{h},q)_{\Omega}+((u_{h})_{x},q_{x})_{\Omega}+B_{2}(u_{h},q)=0,
\end{align}
where
\begin{align}
B_{1}(v_{h},p)=&\sum\limits_{j=1}^{N}\left({v_{h}^{-}p_{x}^{-}|_{\jr}
-v_{h}^{+}p_{x}^{+}|_{\jl}+\widetilde{v_{x}}p^{-}|_{\jr}
-\widetilde{v_{x}}p^{+}|_{\jl}
}\right.\nonumber\\
&\left.{-\widehat{v} p_{x}^{-}|_{\jr}+ \widehat{v}p_{x}^{+}|_{\jl}}
\right),\label{B1}\\
B_{2}(u_{h},q)=&\sum\limits_{j=1}^{N}\left({-u_{h}^{-}q_{x}^{-}|_{\jr}+u_{h}^{+}q_{x}^{+}|_{\jl}-\widetilde{u_{x}}q^{-}|_{\jr}+\widetilde{u_{x}}q^{+}|_{\jl}
}\right.\nonumber\\
&\left.{+\widehat{u}q_{x}^{-}|_{\jr}-\widehat{u}q_{x}^{+}|_{\jl}}\right).\label{B2}
\end{align}
Then we take $p=u_{h}$,~$s=-w_{h}$ and $q=v_{h}$ and add the
three equalities (\ref{shuadd1})-(\ref{shuadd3}) to obtain
\begin{align}
\frac{1}{2}\frac{d}{dt}\int_{\Omega}u_{h}^{2}(x,t)dx+\int_{\Omega}b(u_{h})w_{h}^{2}(x,t)dx+B_{1}(v_{h},u_{h})+B_{2}(u_{h},v_{h})=0.
\end{align}
However,
\begin{align}\label{B1+B2}
&~~~B_{1}(v_{h},u_{h})+B_{2}(u_{h},v_{h})\nonumber\\
&=\sum\limits_{j=1}^{N}\left( {v_{h}^{-}(u_{h})_{x}^{-}-v_{h}^{+}(u_{h})_{x}^{+}+\widetilde{v_{x}}u_{h}^{-}-\widetilde{v_{x}}u_{h}^{+}
-\widehat{v}({u}_{h})_{x}^{-}+\widehat{v}({u}_{h})_{x}^{+}}\right.\nonumber\\
&~~~\left.{-u_{h}^{-}(v_{h})_{x}^{-}+u_{h}^{+}(v_{h})_{x}^{+}-\widetilde{u_{x}}v_{h}^{-}
+\widetilde{u_{x}}v_{h}^{+}+\widehat{u}(v_{h})_{x}^{-}-\widehat{u}(v_{h})_{x}^{-}}\right)|_{\jl}\nonumber\\
&=0,
\end{align}
for all of our flux choices (\ref{flux_linear_forth1})-(\ref{flux4}). Then we have (\ref{stability_equation1}).
\end{proof}

\subsection{Error estimates}

In this subsection, we state the error estimates of our scheme in the
linear case, namely $b(u)=1$.  In this case, (\ref{linear_forth_scheme3})
in the scheme becomes a trivial statement $v_h=w_h$.

\begin{thm}\label{errorestimat_thm1}
Let $u$ be the exact solution of equation (\ref{linear_forth}) with
$b(u)=1$, and $w=u_{xx}$, which are sufficiently smooth with bounded
derivatives. Let $u_{h}$ and $w_{h}$ be solutions of
(\ref{linear_forth_scheme1}), (\ref{linear_forth_scheme2}), with
any choice of fluxes (\ref{flux_linear_forth1})-(\ref{flux_linear_forth4}), and let $V_{h}$ be the space of
piecewise polynomials $\mathcal{P}^{k},~k\geq 1$, then we have the
following error estimate:
\begin{align}\label{errorestimate_equation1}
\|u(t)-u_{h}(t)\|+\int_{0}^{t}\|w(t)-w_{h}(t)\|dt\leq Ch^{k+1},
\end{align}
where $C$ is a constant independent of $h$ and dependent
on $\|u\|_{k+3}$, and on $t$.
\end{thm}
\begin{proof}
Without loss of generality, we choose the flux (\ref{flux_linear_forth1}).
Let
$$e_{u}=u-u_{h}, \quad e_{w}=w-w_{h}$$
 be the errors between the
numerical and exact solutions. Since $u$ and $w$ clearly satisfy
the scheme (\ref{linear_forth_scheme1}) and (\ref{linear_forth_scheme2})
as well, we can obtain the cell error equations: for all $p,~q\in V_{h}$
\begin{align}
&((e_{u})_{t},p)_{j}+(e_{w},p_{xx})_{j}+(e_{w})_{x}^{-}p^{-}|_{\jr}-(e_{w})_{x}^{-}p^{+}|_{\jl}
-e_{w}^{-}p_{x}^{-}|_{\jr}+e_{w}^{-}p_{x}^{+}|_{\jl}=0,\label{errorequation11}\\
&(e_{w},q)_{j}-(e_{u},q_{xx})_{j}-(e_{u})_{x}^{+}q^{-}|_{\jr}+(e_{u})_{x}^{+}q^{+}|_{\jl}
+e_{u}^{+}q_{x}^{-}|_{\jr}-e_{u}^{+}q_{x}^{+}|_{\jl}=0.\label{errorequation12}
\end{align}
Since $k\geq 1$, we can choose a projection $P_{1h}^{\pm}$ defined in (\ref{projection1.1}) and (\ref{projection1.2}).
Denote
\begin{align*}
\eta_{u}=u-P_{1h}^{+}u,~~\xi_{u}=u_{h}-P_{1h}^{+}u,~~\eta_{w}=w-P_{1h}^{-}w,~~\xi_{w}=w_{h}- P_{1h}^{-}w,
\end{align*}
and take $p=\xi_{w}$ and $q=\xi_{u}$ in (\ref{errorequation11}) and (\ref{errorequation12}) respectively. By the stability and property of projection $P_{1h}^{\pm}$ we have
\begin{align}
((\xi_{u})_{t}, \xi_{u})_{\Omega}+(\xi_{w}, \xi_{w})_{\Omega}=((\eta_{u})_{t}, \xi_{u})_{\Omega}+(\eta_{w}, \xi_{w})_{\Omega}.
\end{align}
Then
\begin{align*}
\frac{d}{dt}\|\xi_{u}\|^{2}+\|\xi_{w}\|^{2}\leq Ch^{k+1}\|\xi_{u}\|
+ Ch^{k+1}\|\xi_{w}\|.
\end{align*}
Next we use Gronwall's inequality and choose $u_{h}(0)=P_{1h}^{+}u(0)$
to obtain
\begin{align*}
\|\xi_{u}\|(t)+\int_{0}^{t}\|\xi_{w}\|dt\leq Ch^{k+1},
\end{align*}
and
\begin{align*}
\|e_{u}\|(t)+\int_{0}^{t}\|e_{w}\|dt\leq \|\xi_{u}\|(t)+\int_{0}^{t}\|\xi_{w}\|dt+\|\eta_{u}\|(t)+\int_{0}^{t}\|\eta_{w}\|dt\leq Ch^{k+1},
\end{align*}
where $C$ is a constant independent of $h$ and dependent on $\|u\|_{k+3}$, $\|u_{t}\|_{k+1}$, $k$ and $t$.
\end{proof}

\section{The fifth order problem}

Next we study the DG method for the following one-dimensional nonlinear
fifth order equation
\begin{align}\label{linear_fifth}
&u_{t}+f(u_{xx})_{xxx}=0, ~~~~(x,t)\in [0,2\pi]\times (0,T],\\
&u(x,0)=u_{0}(x),\quad x\in \mathbb{R},
\end{align}
with periodic boundary conditions, where $u_{0}(x)$ is a smooth function.

\subsection{The numerical scheme}

Similar to the fourth order problem (\ref{linear_forth}), we rewrite (\ref{linear_fifth}) into a system:
\begin{align}
&u_{t}+w_{xx}=0,\label{linear_fifth_system1}\\
&w-f(v)_{x}=0,\label{linear_fifth_system2}\\
&v-u_{xx}=0.\label{linear_fifth_system3}
\end{align}
Then our DG method is defined as follows: find
$u_{h},~w_{h},~v_{h}\in V_{h}$ such that for all $p,~s,~q\in V_{h}$, we have
\begin{align}
&((u_{h})_{t},p)_{j}+(w_{h},p_{xx})_{j}+\widetilde{w_{x}}p^{-}|_{\jr}-\widetilde{w_{x}}p^{+}|_{\jl}
-\widehat{w}p_{x}^{-}|_{\jr}+\widehat{w}p_{x}^{+}|_{\jl}=0,\label{linear_fifth_scheme1}\\
&(w_{h},s)_{j}+(f(v_{h}),s_{x})_{j}
-\widehat{f}s^{-}|_{\jr}+\widehat{f}s^{+}|_{\jl}=0,\label{linear_fifth_scheme2}\\
&(v_{h},q)_{j}-(u_{h},q_{xx})_{j}-\widetilde{u_{x}}q^{-}|_{\jr}+\widetilde{u_{x}}q^{+}|_{\jl}
+\widehat{u}q_{x}^{-}|_{\jr}-\widehat{u}q_{x}^{+}|_{\jl}=0.\label{linear_fifth_scheme3}
\end{align}
Here $\widehat{w}$, $\widetilde{w_{x}}$, $\widehat{f}$,
$\widehat{u}$, $\widetilde{u_{x}}$ are numerical fluxes. We can take either of the following two choices for  these five fluxes
\begin{align}\label{flux_linear_fifth1}
\widehat{w}=w_{h}^{-},~\widetilde{w_{x}}=(w_{h})_{x}^{-},~\widehat{f}=\widehat{f}(v_{h}^{-},v_{h}^{+}),~ \widehat{u}=u_{h}^{+},~ \widetilde{u_{x}}=(u_{h})_{x}^{+},
\end{align}
or
\begin{align}\label{flux_linear_fifth2}
\widehat{w}=w_{h}^{+},~ \widetilde{w_{x}}=(w_{h})_{x}^{+},~\widehat{f}=\widehat{f}(v_{h}^{-},v_{h}^{+}),~ \widehat{u}=u_{h}^{-},~ \widetilde{u_{x}}=(u_{h})_{x}^{-},
\end{align}
where $\widehat{f}(v^{-},v^{+})$ is a monotone flux for $f(v)$. Here monotone flux means that the function $\widehat{f}$ is a non-decreasing function of its first argument and a non-increasing function of its second argument. It is also assumed to be at least Lipschitz continuous with respect to each argument and to be consistent with the physical flux $f(v)$ in the sense that $\widehat{f}(v,v)=f(v)$.

\begin{remark}
It is crucial that $\widehat{w}$ and $\widetilde{u_{x}}$ come from the
opposite sides, $\widetilde{w_{x}}$ and $\widehat{u}$
come from the opposite sides. We have at least four choices of these alternating fluxes or similar fluxes in (\ref{flux4}), as in fourth order case. But here we just give the rule of alternating, and list part of them for simplicity.
\end{remark}
\subsection{Stability analysis}

In this subsection, we will show the stability property of the
scheme (\ref{linear_fifth_scheme1})-(\ref{linear_fifth_scheme3}) with
the choice of fluxes (\ref{flux_linear_fifth1})
or (\ref{flux_linear_fifth2}).

\begin{thm}\label{stability_linear_fifth}
Our scheme (\ref{linear_fifth_scheme1}), (\ref{linear_fifth_scheme2})
and (\ref{linear_fifth_scheme3}) with the choice of fluxes
(\ref{flux_linear_fifth1}) or (\ref{flux_linear_fifth2}) is stable, i.e
\begin{align}\label{stability_equation2}
\frac{1}{2}\frac{d}{dt}\int_{\Omega}u_{h}^{2}(x,t)dx\leq0.
\end{align}
\end{thm}

\begin{proof}
Integrate by parts in the scheme
(\ref{linear_fifth_scheme1}), (\ref{linear_fifth_scheme3}) and sum
over $j$, we obtain
\begin{align}
&((u_{h})_{t},p)_{\Omega}-((w_{h})_{x},p_{x})_{\Omega}+B_{1}(w_{h},p)=0,\\
&(w_{h},s)_{\Omega}+(f(v_{h}),s_{x})_{\Omega}+B_{3}(f,s)=0,\\
&(v_{h},q)_{\Omega}+((u_{h})_{x},q_{x})_{\Omega}+B_{2}(u_{h},q)=0,
\end{align}
where $B_1$ and $B_2$ have been defined before in (\ref{B1}) and (\ref{B2}),
and
\begin{align}\label{B3}
B_{3}(f,s)=&\sum\limits_{j=1}^{N}\left(-\widehat{f}s^{-}|_{\jr}+\widehat{f}s^{+}|_{\jl}
\right).
\end{align}
Then we take $p=u_{h}$, $s=-v_{h}$ and $q=w_{h}$ and add the
three equations to obtain
\begin{align}
&\frac{1}{2}\frac{d}{dt}\int_{\Omega}u_{h}^{2}(x,t)dx-(f(v_{h}),(v_{h})_{x})_{\Omega}
+B_{1}(w_{h},u_{h})+B_{3}(f,-v_{h})+B_{2}(u_{h},w_{h})=0 .
\end{align}
By (\ref{B1+B2}), we have $B_{1}(w_{h},u_{h})+B_{2}(u_{h},w_{h})=0$, then
\begin{align}
&\frac{1}{2}\frac{d}{dt}\int_{\Omega}u_{h}^{2}(x,t)dx+\sum\limits_{j=1}^{N}(\widehat{G}_{j+\frac{1}{2}}-\widehat{G}_{j-\frac{1}{2}}+\Theta_{j-\frac{1}{2}})=0,
\end{align}
where
\begin{align}
\widehat{G}_{j+\frac{1}{2}}&=(-F(v_{h}^{-})+\widehat{f}v_{h}^{-})\Big|_{\jr},\quad F(v_h)=\int^{v_h}f(\tau)d\tau,\\
\Theta_{j-\frac{1}{2}}&=(F(v_{h}^{+})-F(v_{h}^{-})+\widehat{f}v_{h}^{-}-\widehat{f}v_{h}^{+})\Big|_\jl,
\end{align}
for both of our flux choices (\ref{flux_linear_fifth1}) and (\ref{flux_linear_fifth2}).
By the monotonicity of the fluxes $\widehat{f}$ and periodic boundary condition we obtain
\begin{align}\label{theta5}
\Theta_{j-\frac{1}{2}}\geq 0.
\end{align}
Then we have (\ref{stability_equation2}).
\end{proof}
\begin{remark}
We can also choose the central flux for nonlinear term $f(v)$
$$
\widehat{f}_\jl=\frac{F(v^+_h)-F(v^-_h)}{v^+_h-v^-_h}\Big|_\jl,
$$
then our scheme will be conservative, that means
$\Theta_{j-\frac{1}{2}}=0$ in (\ref{theta5})
and
$$\frac{d}{dt}\int_{\Omega}u_{h}^{2}(x,t)dx=0.$$
\end{remark}

\subsection{Error estimates}

In this subsection we consider the linear case, $f(v)=v$. Then we have the following optimal error estimate:

\begin{thm}\label{errorestimat_thm2}
Let $u$ be the exact solution of equation (\ref{linear_fifth}) with
$f(v)=v$, and $w=u_{xxx}$, $v=u_{xx}$, which are sufficiently smooth with
bounded derivatives. Let $u_{h}$, $v_{h}$, $w_{h}$ be the numerical
solutions obtained from the scheme
(\ref{linear_fifth_scheme1})-(\ref{linear_fifth_scheme3}) with the choice of fluxes
(\ref{flux_linear_fifth1}) or (\ref{flux_linear_fifth2}) and $\widehat{f}(v)=v^-$. If we
use the $V_{h}$ space with piecewise polynomials $\mathcal{P}^{k},~k\geq 1$,
then we have the following error estimate:
\begin{align}\label{errorestimate_equation2}
\|u(t)-u_{h}(t)\|+\|v(t)-v_{h}(t)\|+\|w(t)-w_{h}(t)\|\leq Ch^{k+1},
\end{align}
where $C$ is a constant independent of $h$ and dependent on $\|u\|_{k+4}$, $\|u_{t}\|_{k+1}$, $k$ and $t$.
\end{thm}
To prove Theorem \ref{errorestimat_thm2} we need some lemmas,
addressing the relationship between the derivative and the element
interface jump of the numerical solution and the auxiliary variable
numerical solution of the derivative.
This plays an important role in the error estimates analysis.
Firstly, we have Lemma \ref{fifth_lemmav}, which was proved
in \cite{Wangimex2015} for the LDG method and extended to the
multi-dimensional case in \cite{Wangimex2016}.

\begin{lemma}\cite{Wangimex2015} \label{fifth_lemmav}
Suppose $(w_{h}, v_{h})\in V_{h}\times V_{h}$ is the solution of
the scheme (\ref{linear_fifth_scheme2}) with $f(v)=v$, then there
exists a positive constant $C$ which is independent of $h$,
such that $\forall j\in Z_{N}$
\begin{align}
\|(v_{h})_{x}\|_{I_{j}}+h^{-\frac{1}{2}}|\jump{v_{h}}|_{\jl}\leq C\|w_{h}\|_{I_{j}}.\label{fifth_vxjump}
\end{align}
\end{lemma}
\smallskip

Next, we establish similar results for $w_{h}$ in the equation (\ref{linear_fifth_scheme1}) as in \cite{Wangimex2015}.
\begin{lemma}\label{fifth_lemma_wxxjump}
Suppose $(u_{h}, w_{h})\in V_{h}\times V_{h}$ is the solution of
the scheme (\ref{linear_fifth_scheme1}), then there exists a positive
constant $C$ which is independent of  $h$, such
that $\forall j\in Z_{N}$
\begin{align}
\|(w_{h})_{xx}\|_{I_{j}}+h^{-\frac{1}{2}}|\jump{(w_{h})_{x}}|_{\jr}
+h^{-\frac{3}{2}}|\jump{w_{h}}|_{\jr}\leq C\|(u_{h})_{t}\|_{I_{j}}.\label{fifth_wxxjump}
\end{align}
\end{lemma}
\begin{proof}
Without loss of generality, we choose the flux (\ref{flux_linear_fifth2})
\begin{align*}
\widehat{w}=w_{h}^{+},~ \widetilde{w_{x}}=(w_{h})_{x}^{+},~\widehat{f}=v^{-},~ \widehat{u}=u_{h}^{-},~ \widetilde{u_{x}}=(u_{h})_{x}^{-}.
\end{align*}
Recalling the equation (\ref{linear_fifth_scheme1}), after integration
by parts we have
\begin{align}
((u_{h})_{t},p)_{j}+((w_{h})_{xx},p)_{j}-{\jump{w_{h}}}_{\jr}(p_{x})^{-}_{\jr}+{\jump{(w_{h})_{x}}}_{\jr}p^{-}_{\jr}=0.\label{linear_fifth_scheme1-byparts}
\end{align}
Let $L_{k}$ be the standard Legendre polynomial of degree $k$ in $[-1,1]$, we have $L_{k}(1)=1$ and $L_{k}$ is orthogonal to any polynomials with degree at most $k-1$. First we take
$$p(x)|_{I_{j}}=(w_{h})_{xx}(x)+AL_{k}(\xi)+BL_{k-1}(\xi),$$
in (\ref{linear_fifth_scheme1}), with $\displaystyle \xi=\frac{2(x-x_{j})}{h_{j}}$
$$A=-\frac{h_{j}(w_{h})_{xxx}(x_{\jr}^{-})}{2k}+\frac{L_{k-1}^{'}(1)(w_{h})_{xx}(x_{\jr}^{-})}{k},$$ and $$B=\frac{h_{j}(w_{h})_{xxx}(x_{\jr}^{-})}{2k}-\frac{L_{k-1}^{'}(1)(w_{h})_{xx}(x_{\jr}^{-})}{k}-(w_{h})_{xx}(x_{\jr}^{-}),$$
$p(x)\in V_{h}$ and is well defined since $k\geq1$ in our function space. Clearly, there hold $p(x_{\jr}^{-})=0$, $p_{x}(x_{\jr}^{-})=0$, and $((w_{h})_{xx},p)_{j}=((w_{h})_{xx},(w_{h})_{xx})_{j}$. By (\ref{linear_fifth_scheme1-byparts}) we have
$$((u_{h})_{t},p)_{j}+((w_{h})_{xx},(w_{h})_{xx})_{j}=0.$$
Thus
\begin{align*}
\|(w_{h})_{xx}\|_{j}^{2}&\leq\|(u_{h})_{t}\|_{j}\left(\|(w_{h})_{xx}\|_{j}+|A|\|L_{k}(\xi)\|_{j}+|B|\|L_{k-1}(\xi)\|_{j}\right)\\
&\leq C\|(u_{h})_{t}\|_{j}\|(w_{h})_{xx}\|_{j},
\end{align*}
where the first inequality is obtained by using the Cauchy-Schwartz inequality and the second is derived by using the inverse inequality and the fact
$\|L_{k}(\xi)\|_{j}\leq Ch^{\frac{1}{2}} $.
Therefore,
\begin{align}
\|(w_{h})_{xx}\|_{j}\leq C\|(u_{h})_{t}\|_{j}.\label{fifth_wxx}
\end{align}
Next we take $p=1$ in (\ref{linear_fifth_scheme1-byparts}) to obtain
$$((u_{h})_{t},1)_{j}+((w_{h})_{xx},1)_{j}+{\jump{(w_{h})_{x}}}_{\jr}=0,$$
then, by (\ref{fifth_wxx}) and the Cauchy-Schwartz inequality we get
\begin{align}
|{\jump{(w_{h})_{x}}}_{\jr}|\leq h^{\frac{1}{2}}\left(\|(u_{h})_{t}\|_{j}+\|(w_{h})_{xx}\|_{j}\right)\leq C h^{\frac{1}{2}}\|(u_{h})_{t}\|_{j}.\label{fifth_jumpwx}
\end{align}
Our next choice of the test function is $p=\xi$
in (\ref{linear_fifth_scheme1-byparts}), which gives
$$((u_{h})_{t},\xi)_{j}+((w_{h})_{xx},\xi)_{j}-\frac{2}{h_{j}}
{\jump{w_{h}}}_{\jr}+{\jump{(w_{h})_{x}}}_{\jr}=0.
$$
By (\ref{fifth_wxx}), (\ref{fifth_jumpwx}) and the Cauchy-Schwartz inequality we get
\begin{align}
|{\jump{w_{h}}}_{\jr}|&\leq Ch^{\frac{3}{2}}\left(\|(u_{h})_{t}\|_{j}+\|(w_{h})_{xx}\|_{j}\right)\leq Ch^{\frac{3}{2}}\|(u_{h})_{t}\|_{j}.\label{fifth_jumpw}
\end{align}
Finally, we get the desired result (\ref{fifth_wxxjump}).
\end{proof}

Based on the relationship constructed in the Lemma \ref{fifth_lemmav}
and Lemma \ref{fifth_lemma_wxxjump}, we can easily use the discrete
Poincar\'{e} inequalities \cite{Susannediscrete, Susannesiam} to
estimate $w_{h}$ and $v_{h}$.

\begin{lemma}\label{fifth_lemmawx}
Let $(u_{h}, v_{h},w_{h})\in V_{h}$ be the solutions of the
scheme (\ref{linear_fifth_scheme1})-(\ref{linear_fifth_scheme3}), then
there exists a positive constant $C$ which are independent of $h$, such that
\begin{align}
\|(w_{h})_{x}\|&\leq C\|(u_{h})_{t}\|,\label{fifth_wx}\\
\|w_{h}\|&\leq C\|(u_{h})_{t}\|,\label{fifth_w}\\
\|v_{h}\|&\leq C\|w_{h}\|.\label{fifth_v}
\end{align}
\end{lemma}

With all these preparations, we can start the proof
of Theorem \ref{errorestimat_thm2}.

\begin{proof} (\begin{small}\textbf{The proof of Theorem \ref{errorestimat_thm2}}\end{small})

Without loss of generality, we choose the flux (\ref{flux_linear_fifth2}).
Let
$$e_{u}=u-u_{h}, \quad e_{v}=v-v_{h}, \quad  e_{w}=w-w_{h}$$
be the errors
between the numerical and exact solutions.
Since $u$, $v$ and $w$ clearly satisfy (\ref{linear_fifth_scheme1})-(\ref{linear_fifth_scheme3}) we can obtain the cell error equations: for all $p,~s,~q\in V_{h}$
\begin{align}
&((e_{u})_{t},p)_{j}+(e_{w},p_{xx})_{j}+(e_{w})_{x}^{+}p^{-}|_{\jr}-(e_{w})_{x}^{+}p^{+}|_{\jl}
-e_{w}^{+}p_{x}^{-}|_{\jr}+e_{w}^{+}p_{x}^{+}|_{\jl}=0,\label{linear_fifth_errorequation1}\\
&(e_{w},s)_{j}+(e_{v},s_{x})_{j}
-e_{v}^{-}s^{-}|_{\jr}+e_{v}^{-}s^{+}|_{\jl}=0,\label{linear_fifth_errorequation2}\\
&(e_{v},q)_{j}-(e_{u},q_{xx})_{j}-(e_{u})_{x}^{-}q^{-}|_{\jr}+(e_{u})_{x}^{-}q^{+}|_{\jl}
+e_{u}^{-}q_{x}^{-}|_{\jr}-e_{u}^{-}q_{x}^{+}|_{\jl}=0.\label{linear_fifth_errorequation3}
\end{align}
Since $k\geq1$ we choose the projections $P_{1h}^{\pm}$, and $P_{h}^{-}$,
which are defined in (\ref{guass_radau_projection})-(\ref{projection1.2}).
Denote
\begin{align*}
&\eta_{u}=u-P_{1h}^{-}u,~~\xi_{u}=u_{h}-P_{1h}^{-}u,\\
&\eta_{w}=w-P_{1h}^{+}w,~~\xi_{w}=w_{h}-P_{1h}^{+}w,\\
&\eta_{v}=v-P_{h}^{-}v,~~\xi_{v}=v_{h}-P_{h}^{-}v.
\end{align*}
Furthermore by the error equations (\ref{linear_fifth_errorequation1})-(\ref{linear_fifth_errorequation3}) and Lemma  \ref{fifth_lemmav}, Lemma \ref{fifth_lemma_wxxjump} and Lemma \ref{fifth_lemmawx} we have
\begin{align}
\|\xi_{w}\|&\leq C\|(e_{u})_{t}\|\leq C\|(\xi_{u})_{t}\|+ Ch^{k+1}, \label{xi_w}\\
\|\xi_{v}\|&\leq C\|e_{w}\|\leq C\|\xi_{w}\|+Ch^{k+1}.\label{xi_v}
\end{align}
\begin{itemize}
\item {\bf Error estimates for the initial condition.}
\end{itemize}
We choose the initial condition $u_{h}(x,0)$ such that
\begin{align}
w_{h}(x,0)=P_{1h}^{+}w(x,0), \quad w(x,0)=u_{xxx}(x,0).
\end{align}
Then we have
\begin{align*}
\|w(x,0)-w_{h}(x,0)\|\leq Ch^{k+1}.
\end{align*}
By (\ref{xi_w}) and (\ref{xi_v}) we get
$$\|\xi_v\|\leq \|\xi_w\|+Ch^{k+1}\leq Ch^{k+1},$$
$$\|\xi_u\|\leq \|\xi_v\|+Ch^{k+1}\leq Ch^{k+1},$$
and we have the following estimates:
\begin{align}\label{fifth_initial}
\|u(x,0)-u_{h}(x,0)\|+\|v(x,0)-v_{h}(x,0)\|+\|w(x,0)-w_{h}(x,0)\|\leq Ch^{k+1}.
\end{align}
Next we choose $t=0$ in (\ref{linear_fifth_errorequation1}), due to the choice of $w_{h}(x,0)$ we have
$$(u_{t}(0)-(u_{h})_{t}(0),p)_{j}=0.$$
Now, we choose $p=(u_{h})_{t}(0)-P(u_{t}(0))$, $P$ is the standard $L^2$ projection, and obtain
\begin{align}\label{fifth_ini_ut}
\|u_{t}(x,0)-(u_{h})_{t}(0)\|\leq Ch^{k+1}.
\end{align}

\begin{itemize}
\item {\bf Error estimates for $t>0$.}
\end{itemize}

Then we take $p=\xi_{u}$, $s=-\xi_{v}$ and $q=\xi_{w}$, and
add the three equations
(\ref{linear_fifth_errorequation1})-(\ref{linear_fifth_errorequation3})
and also sum over $j$.  By the stability and the properties of the
projections we can obtain
\begin{align*}
((\xi_{u})_{t},\xi_{u})_{\Omega}+\sum\limits_{j=1}^{N}{\jump{\xi_{v}}}^{2}_{j-\frac{1}{2}}=((\eta_{u})_{t},\xi_{u})_{\Omega}-(\eta_{w},\xi_{v})_{\Omega}
+(\eta_{v},\xi_{w})_{\Omega}.
\end{align*}

Next, we take the time derivative of the
three error equations
(\ref{linear_fifth_errorequation1})-(\ref{linear_fifth_errorequation3}),
and take $p=(\xi_{u})_{t}$, $s=-(\xi_{v})_{t}$ and $q=(\xi_{w})_{t}$
to obtain
\begin{align*}
((\xi_{u})_{tt},(\xi_{u})_{t})_{\Omega}+\sum\limits_{j=1}^{N}{\jump{(\xi_{v})_{t}}}^{2}_{j-\frac{1}{2}}=((\eta_{u})_{tt},(\xi_{u})_{t})_{\Omega}
-((\eta_{w})_{t},(\xi_{v})_{t})_{\Omega}+((\eta_{v})_{t},(\xi_{w})_{t})_{\Omega}.
\end{align*}
Now, combining the energy equations we get
\begin{align}\label{energyequation}
\frac{1}{2}\frac{d}{dt}(\|\xi_{u}\|^{2}+\|(\xi_{u})_{t}\|^{2})
+\sum\limits_{j=1}^{N}({\jump{\xi_{v}}}^{2}_{j-\frac{1}{2}}
+{\jump{(\xi_{v})_{t}}}^{2}_{j-\frac{1}{2}})
=\Upsilon+\Lambda,
\end{align}
where
\begin{align*}
\Upsilon&=((\eta_{u})_{t},\xi_{u})_{\Omega}-(\eta_{w},\xi_{v})_{\Omega}+(\eta_{v},\xi_{w})_{\Omega}+((\eta_{u})_{tt},(\xi_{u})_{t})_{\Omega},\\
\Lambda&=-((\eta_{w})_{t},(\xi_{v})_{t})_{\Omega}+((\eta_{v})_{t},(\xi_{w})_{t})_{\Omega}.
\end{align*}
By (\ref{xi_w}), (\ref{xi_v}) we have the estimate
$$\|\xi_{v}\|\leq C\|\xi_{w}\|+Ch^{k+1}, \quad \|\xi_{w}\|\leq C\|(\xi_{u})_{t}\|+Ch^{k+1},$$
then we can easily get
\begin{align*}
\Upsilon\leq Ch^{k+1}\|\xi_{u}\|+Ch^{k+1}\|(\xi_{u})_{t}\|+Ch^{2k+2}.
\end{align*}
Next, integrating $\Lambda$ with respect to time between $0$ and $t$, we can get the following equation after integration by parts:
\begin{align*}
\int_{0}^{t}\Lambda dt=-((\eta_{w})_{t},\xi_{v})_{\Omega}|_{0}^{t}+\int_{0}^{t}((\eta_{w})_{tt},\xi_{v})_{\Omega}dt+((\eta_{v})_{t},\xi_{w})_{\Omega}|_{0}^{t}-
\int_{0}^{t}((\eta_{v})_{tt},\xi_{w})_{\Omega}dt.
\end{align*}
We can easily get the following estimates using the approximation property of the projections and the estimates for the initial condition
\begin{align*}
\left| \int_{0}^{t}\Lambda dt \right|
&\leq Ch^{2k+2}+\|\xi_{v}\|^{2}+\|\xi_{w}\|^{2}+\int_{0}^{t}(\|\xi_{v}\|^{2}+\|\xi_{w}\|^{2})dt\\
&\leq Ch^{2k+2}+Ch^{k+1}\int_{0}^{t}\|(\xi_{u})_{t}\|dt.
\end{align*}
Now we integrate (\ref{energyequation}) with respect to the time between $0$ to $t$, using the Cauchy-Schwartz inequality and (\ref{fifth_initial}), (\ref{fifth_ini_ut}) to obtain
\begin{align*}
\frac{1}{2}(\|\xi_{u}\|^{2}+\|(\xi_{u})_{t}\|^{2})\leq \frac{1}{4}\int_{0}^{t}\|\xi_{u}\|^{2}+\|(\xi_{u})_{t}\|^{2}dt+Ch^{2k+2}.
\end{align*}
After employing the Gronwall's inequality, we get
\begin{align*}
\max\limits_{t}\|\xi_{u}\|+\max\limits_{t}\|(\xi_{u})_{t}\|\leq Ch^{k+1} ,
\end{align*}
and also
\begin{align*}
\max\limits_{t}\|\xi_{w}\|+\max\limits_{t}\|\xi_{v}\|\leq Ch^{k+1}.
\end{align*}
After using the standard approximation results, we can
get (\ref{errorestimate_equation2}).
\end{proof}

\section{Extension to high order equations}

The DG method introduced in the previous sections as well as
the theoretical analysis for the stability and error estimates
can be extended to more general high order PDEs, and to multidimensional
cases. Firstly, we consider the extension to the general high order equations,
\begin{align}
u_{t}+(-1)^{[\frac{n}{2}]}u_{x}^{n}=0,
\end{align}
with $n$ being any positive integer. Here $u_{x}^{n}$ denotes
the $n$-th derivative of $u$ with respect to $x$, and
$[\frac{n}{2}]$ is the integer part of $\frac{n}{2}$.

{
In the first two subsections, we will give two
specific examples to introduce our scheme to sixth and seventh order
equations. Then we will summarize to the general case.}

\subsection{Extension to sixth order equations}

In this subsection, we will consider the sixth order equation:
\begin{align}\label{linear_sixth}
&u_{t}-u^{(6)}_{x}=0, ~~~~(x,t)\in [0,2\pi]\times (0,T],\\
&u(x,0)=u_{0}(x),~~~~ x\in \mathbb{R},
\end{align}
where $u_{0}(x)$ is a smooth function, as an example of even order
diffusive equations. For simplicity of discussion, we will again
only consider the periodic boundary conditions.
Firstly, we rewrite the sixth order equation into a system of
third order equations
\begin{align}\label{linear_sixth_system}
&u_{t}-w_{xxx}=0,\\
&w-u_{xxx}=0.
\end{align}
Then our DG method is defined as follows: find
$u_{h},~w_{h}\in V_{h}$ such that for all $p,~q\in V_{h}$, we have
\begin{align}
&((u_{h})_{t},p)_{j}+(w_{h},p_{xxx})_{j}-\widetilde{w_{xx}}p^{-}|_{\jr}+\widetilde{w_{xx}}p^{+}|_{\jl}
+\widetilde{w_{x}}p_{x}^{-}|_{\jr}-\widetilde{w_{x}}p_{x}^{+}|_{\jl}\nonumber\\
&-\widetilde{w}p_{xx}^{-}|_{\jr}+\widetilde{w}p_{xx}^{+}|_{\jl}=0,\label{linear_sixth_scheme1}\\
&(w_{h},q)_{j}+(u_{h},q_{xxx})_{j}-\widetilde{u_{xx}}q^{-}|_{\jr}+\widetilde{u_{xx}}q^{+}|_{\jl}
+\widetilde{u_{x}}q_{x}^{-}|_{\jr}-\widetilde{u_{x}}q_{x}^{+}|_{\jl}\nonumber\\
&-\widehat{u}q_{xx}^{-}|_{\jr}+\widehat{u}q_{xx}^{+}|_{\jl}=0.\label{linear_sixth_scheme2}
\end{align}
Here $\widetilde{w}$, $\widetilde{w_{x}}$,
$\widetilde{w_{xx}}$, $\widetilde{u_{x}}$, $\widetilde{u_{x}}$,
and $\widetilde{u_{xx}}$ are the numerical fluxes. The terms
involving these numerical fluxes appear from repeated integration by parts.
We can take either of the following two choices for these six fluxes
\begin{align}\label{flux_linear_sixth1}
\widetilde{w}\!=\!w_{h}^{-},~ \widetilde{w_{x}}\!=\!(w_{h})_{x}^{-},~
\widetilde{w_{xx}}\!=\!(w_{h})_{xx}^{-},~
\widehat{u}\!=\!u_{h}^{+},~ \widehat{u_{x}}\!=\!(u_{h})_{x}^{+},~
\widehat{u_{xx}}\!=\!(u_{h})_{xx}^{+},
\end{align}
or
\begin{align}\label{flux_linear_sixth2}
\widetilde{w}\!=\!w_{h}^{+},~ \widetilde{w_{x}}\!=\!(w_{h})_{x}^{+},~
\widetilde{w_{xx}}\!=\!(w_{h})_{xx}^{+},~
\widehat{u}\!=\!u_{h}^{-},~ \widehat{u_{x}}\!=\!(u_{h})_{x}^{-},~
\widehat{u_{xx}}\!=\!(u_{h})_{xx}^{-}.
\end{align}
It is crucial that we take
the pair $\widehat{u}$ and $\widetilde{w_{xx}}$ from opposite
sides, the pair $\widehat{u_{x}}$ and $\widetilde{w_{x}}$
from opposite
sides, and the pair $\widehat{u_{xx}}$ and $\widetilde{w}$
from opposite sides.

\begin{thm}\label{stability_linear_forth2}
\textbf{(Stability)}
Our scheme (\ref{linear_sixth_scheme1})-(\ref{linear_sixth_scheme2}) with
the choice of fluxes (\ref{flux_linear_sixth1}) or (\ref{flux_linear_sixth2}) is $L^{2}$ stable, i.e.
\begin{align}\label{stability_equation3}
\frac{1}{2}\frac{d}{dt}\int_{\Omega}u_{h}^{2}(x,t)dx+\int_{\Omega}w_{h}^{2}(x,t)dx=0.
\end{align}
\end{thm}
\begin{proof}
Integrating by parts in the scheme
(\ref{linear_sixth_scheme1})-(\ref{linear_sixth_scheme2}) and summing
over $j$, we have
\begin{align}
&((u_{h})_{t},p)_{\Omega}-((w_{h})_{xxx},p)_{\Omega}+B_{4}(w_{h},p)=0,
\label{shuadd10}\\
&(w_{h},q)_{\Omega}+(u_{h},q_{xxx})_{\Omega}+B_{5}(u_{h},q)=0,
\label{shuadd11}
\end{align}
where
\begin{align}
B_{4}(w_{h},p)=&\sum\limits_{j=1}^{N}\left({
w_{h}^{-}p_{xx}^{-}|_{\jr}-w_{h}^{+}p_{xx}^{+}|_{\jl}
-(w_{h})_{x}^{-}p_{x}^{-}|_{\jr}+(w_{h})_{x}^{+}p_{x}^{+}|_{\jl}
}\right.\nonumber\\
&\left.{
+(w_{h})_{xx}^{-}p^{-}|_{\jr}-(w_{h})_{xx}^{+}p^{+}|_{\jl}
-\widetilde{w_{xx}}p^{-}|_{\jr}+\widetilde{w_{xx}}p^{+}|_{\jl}
}\right.\nonumber\\
&\left.{
+\widetilde{w_{x}}p_{x}^{-}|_{\jr}-\widetilde{w_{x}}p_{x}^{+}|_{\jl}
-\widetilde{w}p_{xx}^{-}|_{\jr}+\widetilde{w}
p_{xx}^{+}|_{\jl}}\right), \label{B4} \\
B_{5}(u_{h},q)=&\sum\limits_{j=1}^{N}\left({
-\widehat{u_{xx}}q^{-}|_{\jr}+\widehat{u_{xx}}q^{+}|_{\jl}
+\widehat{u_{x}}q_{x}^{-}|_{\jr}-\widehat{u_{x}}q_{x}^{+}|_{\jl}
}\right.\nonumber\\
&\left.{-\widehat{u}q_{xx}^{-}|_{\jr}+\widehat{u}
q_{xx}^{+}|_{\jl}}\right).
\label{B5}
\end{align}
Then we take $p=u_{h}$ and $q=w_{h}$ and add the two equations
(\ref{shuadd10})-(\ref{shuadd11}) to obtain
\begin{align}
\frac{1}{2}\frac{d}{dt}\int_{\Omega}u_{h}^{2}(x,t)dx+\int_{\Omega}w_{h}^{2}(x,t)dx+B_{4}(w_{h},u_{h})+B_{5}(u_{h},w_{h})=0.
\end{align}
We can easily check that
\begin{align*}
&~~~B_{4}(w_{h},u_{h})+B_{5}(u_{h},w_{h})=0,
\end{align*}
for both of our flux choices (\ref{flux_linear_sixth1})
and (\ref{flux_linear_sixth2}). Then we have (\ref{stability_equation3}).
\end{proof}

\begin{thm}\label{errorestimat_thm3}
\textbf{(Error estimates)}
Let $u$ be the exact solution of the equation (\ref{linear_sixth})
and $w=u_{xxx}$, which are sufficiently smooth with bounded derivatives.
Let $u_{h}$ and $w_{h}$ be solutions of the scheme
(\ref{linear_sixth_scheme1})-(\ref{linear_sixth_scheme2}) with
either (\ref{flux_linear_sixth1}) or (\ref{flux_linear_sixth2}) as
the numerical fluxes, and let $V_{h}$ be the space of piecewise
polynomials $\mathcal{P}^{k},~k\geq 2$, then we have the following
error estimate
\begin{align}\label{errorestimate_equation3}
\|u(t)-u_{h}(t)\|+\int_{0}^{t}\|w(t)-w_{h}(t)\|dt\leq Ch^{k+1},
\end{align}
where $C$ is a constant independent of $h$ and dependent on $\|u\|_{k+4}$, and $t$.
\end{thm}
\begin{proof}
The proof is similar to that of Theorem \ref{errorestimat_thm1}.
By using the projection $P_{2h}^{\pm}$ defined in
(\ref{projection2.1})-(\ref{projection2.2}) for $k\geq2$ and then
following the line of proof for Theorem \ref{errorestimat_thm1},
we can easily get the result (\ref{errorestimate_equation3}).
\end{proof}

\subsection{Extension to seventh order equations}

In this subsection, we will give the formulation of the scheme as
well as its theoretical results for the seventh order wave equation
\begin{align}\label{linear_seventh}
&u_{t}-u^{(7)}_{x}=0, ~~~~(x,t)\in [0,2\pi]\times (0,T],\\
&u(x,0)=u_{0}(x),~~~~ x\in \mathbb{R},
\end{align}
where $u_{0}(x)$ is a smooth function, as an example of general
odd order wave equations. As mentioned before, we only consider the
periodic boundary conditions.
Similar to the sixth order equation, firstly, we rewrite
(\ref{linear_seventh}) into a system:
\begin{align}\label{linear_seventh_system}
&u_{t}-w_{xxx}=0,\\
&w-v_{x}=0,\\
&v-u_{xxx}=0.
\end{align}
Then our DG method defined as follows: find $u_{h},~v_{h}, ~w_{h}\in V_{h}$ such that for all $p,~s,~q\in V_{h}$, we have
\begin{align}
&((u_{h})_{t},p)_{j}+(w_{h},p_{xxx})_{j}-\widetilde{w_{xx}}p^{-}|_{\jr}+\widetilde{w_{xx}}p^{+}|_{\jl}
+\widetilde{w_{x}}p_{x}^{-}|_{\jr}-\widetilde{w_{x}}p_{x}^{+}|_{\jl}\nonumber\\
&~~~-\widetilde{w}p_{xx}^{-}|_{\jr}+\widetilde{w}p_{xx}^{+}|_{\jl}=0,\label{linear_seventh_scheme1}\\
&(w_{h},s)_{j}+(v_{h},s_{x})_{j}
-\widehat{v}s^{-}|_{\jr}+\widehat{v}s^{+}|_{\jl}=0,\label{linear_seventh_scheme2}\\
&(v_{h},q)_{j}+(u_{h},q_{xxx})_{j}-\widehat{u_{xx}}q^{-}|_{\jr}+\widehat{u_{xx}}q^{+}|_{\jl}
+\widehat{u_{x}}q_{x}^{-}|_{\jr}-\widehat{u_{x}}q_{x}^{+}|_{\jl}\nonumber\\
&~~~-\widehat{u}q_{xx}^{-}|_{\jr}+\widehat{u}q_{xx}^{+}|_{\jl}=0.\label{linear_seventh_scheme3}
\end{align}
Here $\widetilde{w}$, $\widetilde{w_{x}}$,
$\widetilde{w_{xx}}$, $\widehat{v}$, $\widehat{u}$,
$\widehat{u_{x}}$, $\widehat{u_{xx}}$ are numerical fluxes.
For example, we can take either of the following two choices for these fluxes
\begin{align}\label{flux_linear_seventh1}
\widetilde{w}\!=\!w_{h}^{-},~ \widetilde{w_{x}}\!=\!(w_{h})_{x}^{-},~
\widetilde{w_{xx}}\!=\!(w_{h})_{xx}^{-},~
\widehat{v}\!=\!v_{h}^{-},~
\widehat{u}\!=\!u_{h}^{+},~ \widehat{u_{x}}\!=\!(u_{h})_{x}^{+},~
\widehat{u_{xx}}\!=\!(u_{h})_{xx}^{+},
\end{align}
or
\begin{align}\label{flux_linear_seventh2}
\widetilde{w}\!=\!w_{h}^{+},~ \widetilde{w_{x}}\!=\!(w_{h})_{x}^{+},~
\widetilde{w_{xx}}\!=\!(w_{h})_{xx}^{+},~
\widehat{v}\!=\!v_{h}^{-},~
\widehat{u}\!=\!u_{h}^{-},~ \widehat{u_{x}}\!=\!(u_{h})_{x}^{-},~
\widehat{u_{xx}}\!=\!(u_{h})_{xx}^{-}.
\end{align}
It is crucial that we take $\widehat{v}=v_{h}^{-}$ by upwinding,
the pair $\widehat{u}$ and $\widetilde{w_{xx}}$ from opposite
sides, the pair $\widehat{u_{x}}$ and $\widetilde{w_{x}}$
from opposite
sides, and the pair $\widehat{u_{xx}}$ and $\widetilde{w}$
from opposite sides.

\begin{thm}\label{stability_linear_forth3}
\textbf{(Stability)}
Our scheme (\ref{linear_seventh_scheme1})-(\ref{linear_seventh_scheme3})
with the choice of fluxes (\ref{flux_linear_seventh1})
or (\ref{flux_linear_seventh2}) is stable, i.e.
\begin{align}\label{stability_equationseven}
\frac{1}{2}\frac{d}{dt}\int_{\Omega}u_{h}^{2}(x,t)dx\leq0.
\end{align}
\end{thm}
\begin{proof}
Integrating by parts in the scheme
(\ref{linear_seventh_scheme1})-(\ref{linear_seventh_scheme3})
and summing over $j$, we have
\begin{align}
&((u_{h})_{t},p)_{\Omega}-((w_{h})_{xxx},p)_{\Omega}+B_{4}(w_{h},p)=0,
\label{shuadd15} \\
&(w_{h},s)_{\Omega}+(v_{h},s_{x})_{\Omega}+B_{3}(v_{h},s)=0,
\label{shuadd16} \\
&(v_{h},q)_{\Omega}+(u_{h},q_{xxx})_{\Omega}+B_{5}(u_{h},q)=0,
\label{shuadd17}
\end{align}
where $B_3$, $B_4$ and $B_5$ are defined in
(\ref{B3}), (\ref{B4}) and (\ref{B5}), respectively.
Then we take $p=u_{h}$, $s=-v_{h}$ and $q=w_{h}$ in
(\ref{shuadd15}), (\ref{shuadd16}) and (\ref{shuadd17}) respectively,
add the three equations to obtain
\begin{align}
\frac{1}{2}\frac{d}{dt}\int_{\Omega}u_{h}^{2}(x,t)dx
+\frac{1}{2}\sum\limits_{j=1}^{N}({\jump{v_{h}}})^{2}_{\jl}=0,
\end{align}
for both of our flux choices (\ref{flux_linear_seventh1})
and (\ref{flux_linear_seventh2}). Then we
have (\ref{stability_equationseven}).
\end{proof}

\begin{thm}\label{errorestimat_thmseven}
\textbf{(Error estimates)}
Let $u$ be the exact solution of the equation (\ref{linear_seventh}), and $w=u_{xxxx}$, $v=u_{xxx}$, which are sufficiently smooth with bounded derivatives. Let $u_{h}$, $v_{h}$, $w_{h}$ be the numerical solutions of (\ref{linear_seventh_scheme1})-(\ref{linear_seventh_scheme3}). If we use $V_{h}$ as the
space with piecewise polynomials $\mathcal{P}^{k},~k\geq 2$, then we
have the following error estimate:
\begin{align}\label{errorestimate_equationseven}
\|u(t)-u_{h}(t)\|+\|v(t)-v_{h}(t)\|+\|w(t)-w_{h}(t)\|\leq Ch^{k+1},
\end{align}
where $C$ is a constant independent of $h$ and dependent on $\|u\|_{k+5}$, $\|u_{t}\|_{k+1}$, $k$ and $t$.
\end{thm}

\begin{proof}
The proof is similar to that of Theorem \ref{errorestimat_thm2}
and is thus omitted to save space.
\end{proof}

{
\subsection{Extension to general high order cases}
We have introduced the numerical schemes for sixth and seventh order cases. More generally, we summarize the scheme for any high order case.  The proof of stability and error estimate is similar to the sixth and seventh equations, therefore we just list the results and omit the proof.  Again,
we only consider the periodic boundary conditions.
\subsubsection{General even order case}
Let $n$ be a positive even number, and consider the equation
\begin{align}
u_{t}+(-1)^{\frac{n}{2}}u_{x}^{n}=0.\label{even_equation}
\end{align}
Firstly, we rewrite it into a $\frac{n}{2}$-th order system,
\begin{align}
u_{t}+(-1)^{\frac{n}{2}}w_{x}^{\frac{n}{2}}=0,\\
w-u_{x}^{\frac{n}{2}}=0.
\end{align}
Then our DG method is defined as follows: find
$u_{h},~w_{h}\in V_{h}$ such that for all $p,~q\in V_{h}$, we have
\begin{align}
&((u_{h})_{t},p)_{j}+(w_{h},p_{x}^{\frac{n}{2}})_{j}+\sum_{m=0}^{\frac{n}{2}-1}\left((-1)^{\frac{n}{2}+m}\left(\widetilde{w_x}^{\frac{n}{2}-1-m}(p_x^{m})^{-}|_{\jr}-\widetilde{w_x}^{\frac{n}{2}-1-m}(p_x^{m})^{+}|_{\jl}\right)\right)=0,\label{evenscheme1}\\
&(w_{h},q)_{j}-(-1)^{\frac{n}{2}}(u_{h},q_{x}^{\frac{n}{2}})_{j}+\sum_{m=0}^{\frac{n}{2}-1}\left((-1)^{m+1}\left(\widehat{u_x}^{\frac{n}{2}-1-m}(q_x^{m})^{-}|_{\jr}-\widehat{u_x}^{\frac{n}{2}-1-m}(q_x^{m})^{+}|_{\jl}\right)\right)=0.\label{evenscheme2}
\end{align}
\begin{remark}\label{fluxeven}
We choose alternating fluxes. It is crucial that  we take $\widetilde{w_x}^{\frac{n}{2}-1-m}$ and
$\widehat{u_x}^{m}$ from opposite sides, $m=0,1,\cdots, \frac{n}{2}-1$.
\end{remark}

\begin{thm}\label{stability_linear_even}
\textbf{(Stability)}
Our scheme (\ref{evenscheme1})-(\ref{evenscheme2}) with
the choice of alternating fluxes in Remark \ref{fluxeven} is $L^{2}$ stable, i.e.
\begin{align}\label{stability_equation3b}
\frac{1}{2}\frac{d}{dt}\int_{\Omega}u_{h}^{2}(x,t)dx+\int_{\Omega}w_{h}^{2}(x,t)dx=0.
\end{align}
\end{thm}

\begin{thm}\label{errorestimat_thmeven}
\textbf{(Error estimates)}
Let $u$ be the exact solution of the equation (\ref{even_equation}), and $w=u_{x}^{\frac{n}{2}}$,  which are sufficiently smooth with bounded derivatives. Let $u_{h}$, $w_{h}$ be the numerical solutions of (\ref{evenscheme1})-(\ref{evenscheme2}) with alternating fluxes in Remark \ref{fluxeven}. If we use $V_{h}$ as the
space with piecewise polynomials $\mathcal{P}^{k},~k\geq \frac{n}{2}-1$, then we
have the following error estimate:
\begin{align}\label{errorestimate_equationsevenb}
\|u(t)-u_{h}(t)\|+\int_{0}^{t}\|w(t)-w_{h}(t)\|dt\leq Ch^{k+1},
\end{align}
where $C$ is a constant independent of $h$.
\end{thm}

\subsubsection{General odd order case}
Let $n$ be an odd number, and $n\geq 3$. We consider the following equation:
\begin{align}
u_{t}+u_{x}^{n}=0,\label{equation_odd}
\end{align}
Firstly, we rewrite it into a $(\frac{n-1}{2})$-th order system,
\begin{align}
u_{t}+w_{x}^{\frac{n-1}{2}}=0,\label{equation1_odd}\\
w-v_{x}=0,\label{equation1_oddb}\\
v-u_{x}^{\frac{n-1}{2}}=0.\label{equation1_oddc}
\end{align}
Then our DG method is defined as follows: find
$u_{h},~v_h,~w_{h}\in V_{h}$ such that for all $p,~s,~q\in V_{h}$, we have
\begin{align}
&((u_{h})_{t},p)_{j}+(-1)^{\frac{n-1}{2}}(w_{h},p_{x}^{\frac{n-1}{2}})_{j}+\sum_{m=0}^{\frac{n-3}{2}}\left((-1)^{m}\left(\widetilde{w_x}^{\frac{n-3}{2}-m}(p_x^{m})^{-}|_{\jr}-\widetilde{w_x}^{\frac{n-3}{2}-m}(p_x^{m})^{+}|_{\jl}\right)\right)=0.\label{oddscheme1}\\
&(w_{h},s)_{j}+(v_{h},s_{x})_{j}
-\widehat{v}s^{-}|_{\jr}+\widehat{v}s^{+}|_{\jl}=0,\label{oddscheme2}\\
&(v_{h},q)_{j}-(-1)^{\frac{n-1}{2}}(u_{h},q_{x}^{\frac{n-1}{2}})_{j}+\sum_{m=0}^{\frac{n-3}{2}}\left((-1)^{m+1}\left(\widetilde{u_x}^{\frac{n-3}{2}-m}(q_x^{m})^{-}|_{\jr}-\widetilde{u_x}^{\frac{n-3}{2}-m}(q_x^{m})^{+}|_{\jl}\right)\right)=0.\label{oddscheme3}
\end{align}

\begin{remark}\label{fluxodd}
It is crucial that we take  $\widehat{v}$ by upwinding, the pairs $\widetilde{w_x}^{\frac{n-3}{2}-m}$ and
$\widetilde{u_x}^{m}$ from opposite sides, $m=0,1,\cdots, \frac{n-3}{2}$.
\end{remark}

\begin{thm}\label{stability_linear_odd}
\textbf{(Stability)}
Our scheme (\ref{oddscheme1})-(\ref{oddscheme3})
with the choice of fluxes in Remark \ref{fluxodd} is stable, i.e.
\begin{align}\label{stability_equation_odd}
\frac{1}{2}\frac{d}{dt}\int_{\Omega}u_{h}^{2}(x,t)dx\leq0.
\end{align}
\end{thm}

\begin{thm}\label{errorestimat_thm_odd}
\textbf{(Error estimates)}
Let $u$ be the exact solution of the equation (\ref{equation_odd}), and $v=u_{x}^{\frac{n-1}{2}}$,
$w=v_{x}$, , which are sufficiently smooth with bounded derivatives. Let $u_{h}$, $v_{h}$, $w_{h}$ be the numerical solutions of (\ref{oddscheme1})-(\ref{oddscheme3}) with the choice of fluxes in Remark \ref{fluxodd}. If we use $V_{h}$ as the
space with piecewise polynomials $\mathcal{P}^{k},~k\geq \frac{n-3}{2}$, then we
have the following error estimate:
\begin{align}\label{errorestimate_equationsevenc}
\|u(t)-u_{h}(t)\|+\|v(t)-v_{h}(t)\|+\|w(t)-w_{h}(t)\|\leq Ch^{k+1},
\end{align}
where $C$ is a constant independent of $h$.
\end{thm}

}

\section{Extension to the fourth order equation in multi-dimensional
Cartesian meshes}

In this section, we will extend our DG scheme to multi-dimensional
Cartesian meshes for fourth-order equation, as an example of
multi-dimensional extension of our schemes. Without loss of
generality, we describe
our DG method and prove a priori optimal error estimates in
two dimensions ($d=2$), however all the arguments we present in
our analysis depend on the tensor product structure of the meshes
and can be easily extended to higher dimensions ($d>2$).

Hence, from
now on, we shall restrict ourselves to the following two-dimensional
problem:
\begin{align}\label{2d_fourth}
&u_{t}+\Delta^{2}u=0, ~~~~(\textbf{x},t)\in \Omega\times (0,T],
\end{align}
with the periodic boundary condition and initial condition
$$u(\textbf{x},0)=u_{0}(\textbf{x}),$$
where $u_{0}(\textbf{x})$ is a smooth function of $\textbf{x}=(x,y)$,
$\Omega\in R^{2}$ is a bounded rectangular domain.

\subsection{The numerical scheme}

Firstly, we rewrite the fourth-order equation (\ref{2d_fourth})
into a system of second-order equations,
\begin{align}
&u_{t}+\Delta w=0, \label{2d_fourth_system1}\\
&w-\Delta u=0.\label{2d_fourth_system2}
\end{align}
In order to define our DG method for the system
(\ref{2d_fourth_system1})-(\ref{2d_fourth_system2}), let us introduce
some notations.
Let $\Omega_{h}$ denote a tessellation of $\Omega$ with shape-regular elements $K$, and the union of the boundary face of element
$K\in\Omega_{h}$, denoted as $\partial\Omega=\mathop{\cup}\limits_{K\in\Omega_{h}}\partial{K}$. We denote the diameter of $K$ by $h_{K}$, and set $h=\max\limits_{K}h_{K}$.
The finite element spaces with the mesh $\Omega_{h}$ are of the form
\begin{align*}
W_{h}&=\{\eta\in L^{2}(\Omega): \eta|_{K}\in\mathcal{Q}^{k}(K),\forall K \in \Omega_{h}\},
\end{align*}
where $\mathcal{Q}^{k}(K)$ is the space of tensor product of polynomials of degree at most $k\geq{0}$ on $K\in\Omega_{h}$ in each variable defined on $K$.

Since the approximation space in discontinuous Galerkin methods
consists of piecewise polynomials, we need to have a way of denoting
the value of the approximation on the ``left" and ``right" side of an
element boundary $e$. We give the designation $K_{L}$ for element to
the left side of $e$, and $K_{R}$ for element to the right side of $e$
(We refer to \cite{Yankdv} for a proper definition of ``left" and ``right"
in our context, for rectangular meshes these are the usual left and bottom
directions denoted as ``left" and right and top directions denoted as
``right"). The normal vector $\nu_{L}$ and $\nu_{R}$ on the edge $e$ point exterior to $K_{L}$ and $K_{R}$ respectively. Assuming $\psi$ is a function defined on $K_{L}$ and $K_{R}$, let $\psi^{-}$ denote $(\psi|_{K_{L}})|_{e}$ and $\psi^{+}$ denote $(\psi|_{K_{R}})|_{e}$, the left and right traces,
respectively.
The DG method is defined as following: we seek $u_{h}$ and $w_{h}$ in the finite element space $W_{h}\times W_{h}$, such that for all $p,~q\in W_{h}$ we have
\begin{align}
((u_{h})_{t},p)_{K}+(w_{h},\Delta p)_{K}+\langle \widetilde{\nabla w}\cdot \textbf{n},p\rangle_{\partial K}-\langle \widetilde{w},\nabla p\cdot \textbf{n}\rangle_{\partial K}&=0,\label{2dfourthscheme1}\\
(w_{h},q)_{K}-(u_{h},\Delta q)_{K}-\langle \widehat{\nabla u}\cdot \textbf{n},q\rangle_{\partial K}+\langle \widehat{u},\nabla q\cdot \textbf{n}\rangle_{\partial K}&=0.\label{2dfourthscheme2}
\end{align}
Here \textbf{n} denotes the outward unit vector to $\partial K$, and
\begin{align}
(p,q)_{K}:=\int_{K}p(x,y)q(x,y)dxdy,~~~\langle p,\nabla q\cdot \textbf{n}\rangle=\int_{\partial K}p(x,y)(\nabla q(x,y)\cdot \textbf{n})ds,
\end{align}
for any $p,~q\in H^{1}_{\Omega_{h}}$. To complete the definition of the DG scheme we need to define the numerical fluxes $\widehat{u},~\widehat{\nabla u},~\widetilde{w},~\widetilde{\nabla w}$. We can choose the alternating fluxes
\begin{align}\label{2d_fourth_flux1}
\widehat{u}=u_{h}^{+},~\widehat{\nabla u}=(\nabla u_{h})^{+},~\widetilde{w}=w_{h}^{-},~\widetilde{\nabla w}=(\nabla w_{h})^{-},
\end{align}
or
\begin{align}\label{2d_fourth_flux2}
\widehat{u}=u_{h}^{-},~\widehat{\nabla u}=(\nabla u_{h})^{-},~\widetilde{w}=w_{h}^{+},~\widetilde{\nabla w}=(\nabla w_{h})^{+}.
\end{align}

\subsection{$L^{2}$ stability}
In this subsection, we will prove the DG method defined in
(\ref{2dfourthscheme1})-(\ref{2dfourthscheme2}) for the
fourth-order equation satisfies the following $L^{2}$ stability.
\begin{thm}
The solution given by the DG method defined by
(\ref{2dfourthscheme1})-(\ref{2dfourthscheme2}) satisfies
\begin{align}\label{2dstability}
\frac{1}{2}\frac{d}{dt}\int_{\Omega_{h}}u_{h}^{2}(\textbf{x},t)d\textbf{x}+\int_{\Omega_{h}}w_{h}^{2}(\textbf{x},t)d\textbf{x}=0.
\end{align}
\end{thm}
\begin{proof}
We take the test functions $p=u_{h}$, $q=w_{h}$ in
(\ref{2dfourthscheme1}) and (\ref{2dfourthscheme2}) respectively,
and integrate by parts to obtain
$$
((u_{h})_{t},u_{h})_{K}+(w_{h},w_{h})_{K}+H_{\partial K}(u_{h},w_{h})=0,
$$
where
\begin{align*}
H_{\partial K}(p,q)=&\langle w_{h},\nabla u_{h}\cdot \textbf{n}\rangle_{\partial K}+\langle \widetilde{\nabla w}\cdot \textbf{n},p\rangle_{\partial K}-\langle \widetilde{w},\nabla p\cdot \textbf{n}\rangle_{\partial K}-\langle u_{h},\nabla w_{h}\cdot \textbf{n}\rangle_{\partial K}\\
&-\langle \widehat{\nabla u}\cdot \textbf{n},q\rangle_{\partial K}+\langle \widehat{u},\nabla q\cdot \textbf{n}\rangle_{\partial K}.
\end{align*}
Next we sum over the $K$.  Since
\begin{align}
H_{\partial K_{1}\cap e}(u_{h},w_{h})+H_{\partial K_{2}\cap e}(u_{h},w_{h})=0,
\end{align}
with the numerical flux (\ref{2d_fourth_flux1}) or (\ref{2d_fourth_flux2}),
here we suppose $e$ is an inter-element face shared with the
elements $K_{1}$ and $K_{2}$, we can immediately get the $L^{2}$-stability
result (\ref{2dstability}).
\end{proof}

\subsection{Error estimates}

In this subsection, we obtain a priori error estimates for the
approximation $(u_{h},w_{h})$ given by the DG scheme
(\ref{2dfourthscheme1})-(\ref{2dfourthscheme2}).
{
The proof of optimal error estimate in the multi-dimensional case is
different from that in the one-dimensional case, in the definition
and analysis of suitable projections.  Since the projection terms
in the error equations do not vanish as in the one-dimensional case,
we need to obtain certain superconvergence properties of the projections
to deal with these terms.}

\begin{thm}\label{errorestimate_thm2dfourth}
Let $u$ be the solution of the equation (\ref{2d_fourth}) with
periodic boundary condition, and $w=\Delta u$.
Let $u_{h}$ and $w_{h}$ be the numerical solution of the
DG scheme (\ref{2dfourthscheme1})-(\ref{2dfourthscheme2}).
If we use $W_{h}$ as the
space with piecewise polynomials $\mathcal{Q}^{k},~k\geq 1$.
Then for Cartesian meshes, we have
\begin{align*}
\|u(t)-u_{h}(t)\|+\int_{0}^{t}\|w(t)-w_{h}(t)\|dt\leq Ch^{k+1} .
\end{align*}
Here $C$ depends on $\|u\|_{L^{\infty}((0,T);W^{2k+6,\infty})}$,
$\|u_{t}\|_{L^{\infty}((0,T);W^{k+1,\infty})}$, and on $t$,
but is independent of $h$.
\end{thm}

\subsection{Proof of the error estimates}
In this subsection we prove Theorem \ref{errorestimate_thm2dfourth}
stated in the previous section. To do that, firstly, we define the
special projection in Cartesian meshes, similar to the Gauss-Radau
projections in Cartesian meshes \cite{cockburnlocal2001, meng, Xuoptimal}.

On a rectangle $K_{i,j}=I_{i}\times J_{j}$, for $u\in W^{1,\infty}(\overline{K})$, we define
\begin{align}\label{2dprojection}
\Pi^{\pm}u:=P_{1hx}^{\pm}\otimes P_{1hy}^{\pm}u,
\end{align}
with the subscripts indicating the application of the one-dimensional
operators $P_{1h}^{\pm}$ with respect to the corresponding variable.
To be more specific, we shall list explicitly the formulations
for $\Pi^{-}u$, on a rectangular element
$K_{i,j}=I_{i}\times J_{j}:=(x_{\il},x_{\ir})\times (y_{\jl},y_{\jr})$.
We have
\begin{subequations}
\label{shuadd20}
\begin{align}
\int_{K_{i,j}}\Pi^{-}u(x,y)v_{h}(x,y)dxdy&=\int_{K_{i,j}}u(x,y)v_{h}(x,y)dxdy,\label{2dprojection1}\\
\int_{I_{i}}\Pi^{-}u(x,y_{\jr}^{-})v_{h}(x,y_{\jr}^{-})dx&=\int_{I_{i}}u(x,y_{\jr}^{-})v_{h}(x,y_{\jr}^{-})dx,\label{2dprojection2}\\
\int_{I_{i}}(\Pi^{-}u)_{y}(x,y_{\jr}^{-})v_{h}(x,y_{\jr}^{-})dx&=\int_{I_{i}}u_{y}(x,y_{\jr}^{-})v_{h}(x,y_{\jr}^{-})dx,\label{2dprojection3}\\
\int_{J_{j}}\Pi^{-}u(x_{\ir}^{-},y)v_{h}(x_{\ir}^{-},y)dy&=\int_{J_{j}}u(x_{\ir}^{-},y)v_{h}(x_{\ir}^{-},y)dy,\label{2dprojection4}\\
\int_{J_{j}}(\Pi^{-}u)_{x}(x_{\ir}^{-},y)v_{h}(x_{\ir}^{-},y)dy&=\int_{J_{j}}u_{x}(x_{\ir}^{-},y)v_{h}(x_{\ir}^{-},y)dy,\label{2dprojection5}\\
\Pi^{-}u(x_{\ir}^{-},y_{\jr}^{-})&=u(x_{\ir}^{-},y_{\jr}^{-}),\label{2dprojection6}\\
(\Pi^{-}u)_{x}(x_{\ir}^{-},y_{\jr}^{-})&=u_{x}(x_{\ir}^{-},y_{\jr}^{-}),\label{2dprojection7}\\
(\Pi^{-}u)_{y}(x_{\ir}^{-},y_{\jr}^{-})&=u_{y}(x_{\ir}^{-},y_{\jr}^{-}),\label{2dprojection8}\\
(\Pi^{-}u)_{xy}(x_{\ir}^{-},y_{\jr}^{-})&=u_{xy}(x_{\ir}^{-},y_{\jr}^{-}),\label{2dprojection9}
\end{align}
\end{subequations}
for all $v_{h}\in \mathcal{Q}^{k-2}(K)$  and $K\in \Omega_{h}$. Similarly, we can define the projection $\Pi^{+}$.
Existence and the optimal approximation property of the projection $\Pi^{\pm}$ are established in the following lemma.
\begin{lemma}
Assume u is sufficiently smooth, then there exists a
unique $\Pi^{-}u\in W_{h}$, satisfying (\ref{shuadd20}).
Moreover, there holds the following approximation property
\begin{align*}
\|v-\Pi^{\pm}v\|_{L^{2}(K)}+h\|v-\Pi^{\pm}v\|_{H^{1}(K)}\leq Ch^{k+1}\|u\|_{H^{k+1}(K)}.
\end{align*}
\end{lemma}
\begin{proof}
Assume that $u\equiv0$, then by (\ref{2dprojection2}), (\ref{2dprojection6}) and (\ref{2dprojection7}) we have
$$\Pi^{-}u(x,y_{\jr}^{-})=0.$$
Furthermore, by (\ref{2dprojection3}), (\ref{2dprojection8}) and (\ref{2dprojection9}) we get
$$(\Pi^{-}u)_{y}(x,y_{\jr}^{-})=0.$$
Similarly, we have
$\Pi^{-}u(x_{\ir}^{-},y)=0,$
and $(\Pi^{-}u)_{x}(x_{\ir}^{-},y)=0,$
then we obtain
$$
\Pi^{-}u=(x-x_{\ir}^{-})^{2}(y-y_{\jr}^{-})^{2}Q(x,y),~~~~Q(x,y)\in
\mathcal{Q}^{k-2}.
$$
Finally, we take $v_{h}=Q(x,y)$ in (\ref{2dprojection1}) to get
$Q(x,y)\equiv0$, therefore $\Pi^{-}u\equiv0$, and we have finished the proof of the uniqueness and also existence.
Since the one-dimensional operators $P_{1h}^{\pm}$ satisfy
$\|P_{1h}^{\pm}u\|_{L^{\infty}(I_{j})}\leq C \|u\|_{L^{\infty}(I_{j})}$,
similarly
in the two-dimensional case we also have
$\|\Pi^{\pm}u\|_{L^{\infty}(K_{i,j})}\leq C \|u\|_{L^{\infty}(K_{i,j})}$,
here $C$ is a constant independent of $h$.
Again, standard approximation theory implies the optimal approximating
estimates.
\end{proof}

To prove Theorem \ref{errorestimate_thm2dfourth}, firstly we need to
write the error equations. Let
$$e_{u}=u-u_{h}=\eta_{u}-\xi_{u}, \quad e_{w}=w-w_{h}=\eta_{w}-\xi_{w}$$
with
$$\eta_{u}=u-\Pi^{+}u,~\eta_{w}=w-\Pi^{-}w, \quad \xi_{u}=u_{h}-\Pi^{+}u,~\xi_{w}=w_{h}-\Pi^{-}w,$$
then
\begin{align}
((\xi_{u})_{t},p)_{K}+B^{1}_{K}(\xi_{w},p)=&((\eta_{u})_{t},p)_{K}+B^{1}_{K}(\eta_{w},p),\label{2d4errorequation1}\\
(\xi_{w},q)_{K}-B^{2}_{K}(\xi_{u},q)=&(\eta_{w},q)_{K}-B^{2}_{K}(\eta_{u},q)_{K},\label{2d4errorequation2}
\end{align}
where
\begin{align}
B^{1}_{K}(w,p)&=(w,\Delta p)_{K}-\langle w^{-},(\nabla p\cdot \textbf{n})\rangle_{\partial K}+\langle (\nabla w^{-}\cdot \textbf{n}),p\rangle_{\partial K},\label{2D_B1}\\
B^{2}_{K}(u,q)&=(u,\Delta q)_{K}-\langle u^{+},(\nabla q\cdot \textbf{n})\rangle_{\partial K}+\langle (\nabla u^{+}\cdot \textbf{n}),q\rangle_{\partial K}.\label{2D_B2}
\end{align}

Besides the standard approximation results, we will also prove  superconvergence results for the projections $\Pi^{\pm}$ in  Lemma \ref{2D_sup1} and \ref{2Dsupest}. The proof is using similar strategies and skills in \cite{cockburnlocal2001}.

\begin{lemma}\label{2D_sup1}
Let $B^1_{K}(\eta_{w},p)$ and $B^2_{K}(\eta_{u},q)$ be defined
by (\ref{2D_B1}) and (\ref{2D_B2}). Then we have for $k\geq 1$,
\begin{align}
B^{1}_{K}(\eta_{w},p)=0,~ B^{2}_{K}(\eta_{u},q)=0,~\forall u,w \in \mathcal{P}^{k+2}(K),~~p,~q\in \mathcal{Q}^{k}(K).\label{2DBsup}
\end{align}
\end{lemma}
\begin{proof}
The proof of the results for $B^{1}_{K}$ and $B^{2}_{K}$ are analogous; therefore we just prove the one for $B^{2}_{K}(\eta_{u},q)$. Let us
consider the rectangular element $K_{ij}=I_{i}\times J_{j}=(x_{\il},
x_{\ir})\times (y_{\jl},y_{\jr})$.  By the definition of
$B^{2}_{K}(\eta_{u},q)$ we have
\begin{align*}
B^{2}_{K}(\eta_{u},q)=&\int_{K_{i,j}}(u-\Pi^{+}u)(q_{xx}+q_{yy})dxdy\\
&-\int_{y_{\jl}}^{y_{\jr}}(u-\Pi^{+}u)(x_{\ir}^{+},y)q_{x}(x_{\ir}^{-},y)
-(u-\Pi^{+}u)(x_{\il}^{+},y)q_{x}(x_{\il}^{+},y)dy\\
&-\int_{x_{\il}}^{x_{\ir}}(u-\Pi^{+}u)(x,y_{\jr}^{+})q_{y}(x,y_{\jr}^{-})
-(u-\Pi^{+}u)(x,y_{\jl}^{+})q_{y}(x,y_{\jl}^{+})dx\\
&+\int_{y_{\jl}}^{y_{\jr}}(u-\Pi^{+}u)_{x}(x_{\ir}^{+},y)q(x_{\ir}^{-},y)
-(u-\Pi^{+}u)_{x}(x_{\il}^{+},y)q(x_{\il}^{+},y)dy\\
&+\int_{x_{\il}}^{x_{\ir}}(u-\Pi^{+}u)_{y}(x,y_{\jr}^{+})q(x,y_{\jr}^{-})
-(u-\Pi^{+}u)_{y}(x,y_{\jl}^{+})q(x,y_{\jl}^{+})dx.
\end{align*}
Since $\Pi^{+}$ is polynomial preserving operator, (\ref{2DBsup}) holds true for every $u\in\mathcal{Q}^{k}(K)$. Therefore, we have to consider the cases $u(x,y)=x^{k+1}, y^{k+1}, x^{k+2}, y^{k+2}, x^{k+1}y, y^{k+1}x$.

Let us start with $u(x,y)=x^{k+1}$. We have $(u-\Pi^{+}u)_{y}(x,y)=0$, by (\ref{2dprojection6}) and (\ref{2dprojection7}), $u(x_{\ir}^{+},y)=\Pi^{+}u(x_{\ir}^{+},y)$, $u_{x}(x_{\ir}^{+},y)=(\Pi^{+}u)_{x}(x_{\ir}^{+},y)$.
Then
\begin{align*}
&\int_{y_{\jl}}^{y_{\jr}}(u-\Pi^{+}u)(x_{\ir}^{+},y)q_{x}(x_{\ir}^{+},y)
-(u-\Pi^{+}u)(x_{\il}^{+},y)q_{x}(x_{\il}^{+},y)dy=0,\\
&\int_{y_{\jl}}^{y_{\jr}}(u-\Pi^{+}u)_{x}(x_{\ir}^{+},y)q(x_{\ir}^{+},y)
-(u-\Pi^{+}u)_{x}(x_{\il}^{+},y)q(x_{\il}^{+},y)dy=0,
\end{align*}
and $\int_{K_{ij}}(u-\Pi^{+}u)q_{xx}dxdy=0$.
Next we integrate by parts
\begin{align*}
&\int_{K_{i,j}}(u-\Pi^{+}u)q_{yy}dxdy\\
=&\int_{x_{\il}}^{x_{\ir}}(u-\Pi^{+}u)(x,y_{\jr}^{-})q_{y}(x,y_{\jr}^{-})
-(u-\Pi^{+}u)(x,y_{\jl}^{+})q_{y}(x,y_{\jl}^{+})dx.
\end{align*}
Therefore, sum all the parts in the definition of $B^{2}_{K}(\eta_{u},q)$, we have
$$B^{2}_{K}(\eta_{u},q)=0.$$
Next, we consider the case $u(x,y)=x^{k+1}y$, in this case $\Pi^{+} u=P_{1hx}^{+}(x^{k+1})y$, and
$$\int_{K_{ij}}(u-\Pi^{+}u)q_{xx}dxdy=\int_{K_{ij}}y(x^{k+1}-P_{1hx}^{+}(x^{k+1}))q_{xx}dxdy=0,$$
and
\begin{align*}
&\int_{K_{i,j}}(u-\Pi^{+}u)q_{yy}dxdy\\
=&\int_{x_{\il}}^{x_{\ir}}y_{\jr}(x^{k+1}-P_{1hx}^{+}(x^{k+1}))q_{y}(x,y_{\jr}^{-})
-y_{\jl}^{+}(x^{k+1}-P_{1hx}^{+}(x^{k+1}))q_{y}(x,y_{\jl}^{+})dx\\
&-\int_{x_{\il}}^{x_{\ir}}(x^{k+1}-P_{1hx}^{+}(x^{k+1}))q(x,y_{\jr}^{-})
-(x^{k+1}-P_{1hx}^{+}(x^{k+1}))q(x,y_{\jl}^{+})dx .
\end{align*}
Then summing all the parts in the definition of $B^{2}_{K}(\eta_{u},q)$,
we have
$$B^{2}_{K}(\eta_{u},q)=0.$$
The proof of the cases $u(x,y)=y^{k+1}, x^{k+2}, y^{k+2}$ and $u(x,y)=y^{k+1}x$ are analogous. This completes the proof of (\ref{2DBsup}).
\end{proof}
\begin{lemma}\label{2Dsupest}
Let $B^{1}_{K}(\eta_{w},p)$ and $B^{2}_{K}(\eta_{u},q)$ defined by (\ref{2D_B1}) and (\ref{2D_B2}). Then we have
\begin{align}
|B^{1}_{K}(\eta_{w},p)|\leq Ch^{k+2}\|w\|_{W^{2k+4,\infty}(\Omega_{h})}\|p\|_{L^{2}(K)},\\
|B^{2}_{K}(\eta_{u},q)|\leq Ch^{k+2}\|u\|_{W^{2k+4,\infty}(\Omega_{h})}
\|q\|_{L^{2}(K)},
\end{align}
where $p,~q\in\mathcal{Q}^{k}(K)$ and the constant $C$ is independent of $h$.
\end{lemma}
\begin{proof}
On each element $K=I_{i}\times J_{j}$, consider the Taylor expansion of $u$ around $(x_{i},y_{j})$
$$u=Tu+R_{k+3},$$
where
\begin{align*}
Tu&=\sum\limits_{l=0}^{k+2}\sum\limits_{m=0}^{l}\frac{1}{(l-m)!m!}\frac{\partial^{l}u(x_{i},y_{j})}{\partial x^{l-m}\partial y^{m}}(x-x_{i})^{l-m}(y-y_{j})^{m},\\
R_{k+3}&=(k+3)\sum\limits_{m=0}^{k+3}\frac{(x-x_{i})^{k+3-m}(y-y_{j})^{m}}{(k+3-m)!m!}\int_{0}^{1}(1-s)^{k+2}
\frac{\partial^{k+3}u(x_{i}^{s},y_{j}^{s})}{\partial x^{k+3-m}\partial y^{m}}ds
\end{align*}
with $x_{i}^{s}=x_{i}+s(x-x_{i})$, $y_{j}^{s}=y_{j}+s(y-y_{j})$. Clearly, $Tu\in \mathcal{P}^{k+2}$ and by Lemma \ref{2D_sup1} we have
$$B^{2}_{K}(Tu-\Pi^{+}(Tu),q)=0,$$
then we have
$$B^{2}_{K}(\eta_{u},q)=T_{1}+T_{2}+T_{3}+T_{4}+T_{5},$$
where
\begin{align*}
T_{1}&=\int_{K_{ij}}(R_{k+3}-\Pi^{+}R_{k+3})(p_{xx}+p_{yy})dxdy,\\
T_{2}&=-\int_{y_{\jl}}^{y_{\jr}}(R_{k+3}-\Pi^{+}R_{k+3})(x_{\ir}^{+},y)p_{x}(x_{\ir}^{-},y)
-(R_{k+3}-\Pi^{+}R_{k+3})(x_{\il}^{+},y)p_{x}(x_{\il}^{+},y)dy,\\
T_{3}&=-\int_{x_{\il}}^{x_{\ir}}(R_{k+3}-\Pi^{+}R_{k+3})(x,y_{\jr}^{+})p_{y}(x,y_{\jr}^{-})
-(R_{k+3}-\Pi^{+}R_{k+3})(x,y_{\jl}^{+})p_{y}(x,y_{\jl}^{+})dx,\\
T_{4}&=\int_{y_{\jl}}^{y_{\jr}}(R_{k+3}-\Pi^{+}R_{k+3})_{x}(x_{\ir}^{+},y)p(x_{\ir}^{-},y)
-(R_{k+3}-\Pi^{+}R_{k+3})_{x}(x_{\il}^{+},y)p(x_{\il}^{+},y)dy,\\
T_{5}&=\int_{x_{\il}}^{x_{\ir}}(R_{k+3}-\Pi^{+}R_{k+3})_{y}(x,y_{\jr}^{+})p(x,y_{\jr}^{-})
-(R_{k+3}-\Pi^{+}R_{k+3})_{y}(x,y_{\jl}^{+})p(x,y_{\jl}^{+})dx.
\end{align*}
which will be estimated one by one below.
 From the approximation properties of the projection $\Pi^{+}$, we have
\begin{align*}
\|R_{k+3}-\Pi^{+}R_{k+3}\|_{L^{2}(K)}\leq Ch^{k+2}\|R_{k+3}\|_{W^{k+1,\infty}(\Omega_{h})},
\end{align*}
and
\begin{align*}
\|R_{k+3}\|_{W^{k+1,\infty}(\Omega_{h})}=\max\limits_{K}
\|R_{k+3}\|_{W^{k+1,\infty}(K)}\leq Ch^{2}\|u\|_{W^{2k+4,\infty}(\Omega_{h})}.
\end{align*}
Combining the above two estimates, we arrive at
\begin{align}
\|R_{k+3}-\Pi^{+}R_{k+3}\|_{L^{2}(K)}\leq Ch^{k+4}\|u\|_{W^{2k+4,\infty}
(\Omega_{h})}.
\end{align}
Similarly, we have that
\begin{align}
\|R_{k+3}-\Pi^{+}R_{k+3}\|_{H^{1}(K)}\leq Ch^{k+3}\|u\|_{W^{2k+4,\infty}(\Omega_{h})}.
\end{align}
It follows from the Cauchy-Schwartz inequality, and the inverse inequality
that
\begin{align*}
|T_{1}|\leq\|R_{k+3}-\Pi^{+}R_{k+3}\|_{L^{2}(K)}\|q_{xx}\|_{L^{2}(K)}
\leq Ch^{k+2}\|u\|_{W^{2k+4,\infty}(\Omega_{h})}\|q\|_{L^{2}(K)}.
\end{align*}
In order to estimate the remaining terms we need to use
the trace inequality to get
$$\|R_{k+3}-\Pi^{+}R_{k+3}\|_{L^{2}(\partial K)}\leq Ch^{k+\frac{7}{2}}\|u\|_{W^{2k+4,\infty}(\Omega_{h})}$$
and
$$\|R_{k+3}-\Pi^{+}R_{k+3}\|_{H^{1}(\partial K)}\leq Ch^{k+\frac{5}{2}}\|u\|_{W^{2k+4,\infty}(\Omega_{h})}$$
Next, by the Cauchy-Schwartz inequality and the inverse inequality,
we arrive at
\begin{align*}
|T_{2}|\leq \|R_{k+3}-\Pi^{+}R_{k+3}\|_{L^{2}(\partial K)}\|q_{x}\|_{L^{2}
(\partial K)}\leq Ch^{k+2}\|u\|_{W^{2k+4,\infty}(\Omega_{h})}\|q\|_{L^{2}(K)} .
\end{align*}
Analogously, we have that
$$
|T_{m}|\leq Ch^{k+2}\|u\|_{W^{2k+4,\infty}
(\Omega_{h})}\|q\|_{L^{2}(K)},~~~m=3,4,5.
$$
The estimates for $B^1(\eta_{u},q)$ now follows by collecting the
results for $T_{m}$, $m=1,2,3,4,5$ obtained above. The proof of Lemma is thus completed.
\end{proof}

Next, we will use these lemmas to prove our final result, Theorem \ref{errorestimate_thm2dfourth}.

\begin{proof}
\textbf{(The proof of Theorem \ref{errorestimate_thm2dfourth}).}
We take $p=\xi_{u}$ and $q=\xi_{w}$ in the error equations
(\ref{2d4errorequation1})-(\ref{2d4errorequation2}), to obtain
\begin{align*}
((\xi_{u})_{t},\xi_{u})_{\Omega_{h}}+(\xi_{w},\xi_{w})_{\Omega_h}
=((\eta_{u})_{t},\xi_{u})_{\Omega_h}+(\eta_{w},\xi_{w})_{\Omega_h}
+\sum\limits_{K}(B^{1}_{K}(\eta_{w},\xi_{u})-B^{2}_{K}(\eta_{u},\xi_{w})).
\end{align*}
Then by the Cauchy-Schwartz inequality and Lemma \ref{2Dsupest}, we have
\begin{align*}
\frac{1}{2}\frac{d}{dt}\|\xi_{u}\|^{2}+\|\xi_{w}\|^{2}\leq Ch^{k+1}
\|\xi_{u}\|^{2}+Ch^{k+1}\|\xi_{w}\|^{2}.
\end{align*}
Next, by Gronwall's inequality and choosing
$u_{h}(0)=\Pi_{h}^{+}u(0)$, we have
\begin{align*}
\|\xi_{u}\|(t)+\int_{0}^{t}\|\xi_{w}\|(t)dt\leq Ch^{k+1},
\end{align*}
and
\begin{align*}
\|e_{u}\|(t)+\int_{0}^{t}\|e_{w}\|dt\leq \|\xi_{u}\|(t)+\int_{0}^{t}\|\xi_{w}\|dt+\|\eta_{u}\|(t)+\int_{0}^{t}\|\eta_{w}\|dt\leq Ch^{k+1},
\end{align*}
where $C$ is a constant independent of $h$ and dependent on $\|u\|_{W^{2k+6,\infty}}$, $\|u_{t}\|_{W^{k+1,\infty}}$ and $t$.
\end{proof}

\section{Numerical results}

In this section, we present numerical examples to verify our
theoretical convergence properties of the DG method for high order
PDEs.

Firstly, we consider the one-dimensional linear fourth and fifth
order time-dependent equations with the periodic boundary condition
in Examples \ref{example1} and \ref{example2}, respectively.
Time discretization is not our major concern in this paper, hence we use
the spectral deferred correction (SDC) \cite{SDC} time discretization
for its simplicity. Our computation is based on the flux choice
(\ref{flux_linear_forth1}) and (\ref{flux_linear_fifth1}),
respectively. The errors and numerical orders of accuracy for
$P^{k}$ elements with $1\leq k\leq 3$ are listed in
Table \ref{table.1} and Table \ref{table.2}. We observe that our
scheme gives the optimal $(k+1)$-th order of the accuracy when $k\geq 1$.
\begin{ex}\label{example1}(Accuracy test for a linear fourth-order problem.)
We consider the following fourth-order time-dependent problem
\begin{align*}
&u_{t}+u_{xxxx}=0,~~~~~~~(x,t)\in[0,2\pi]\times(0,1],\\
&u(x,0)= \sin(x).
\end{align*}
The exact solution is $$u(x,t)=e^{-t} \sin(x).$$
\end{ex}

\begin{table}[!htb]\centering
\caption{\label{table.1}Errors and the corresponding convergence rates
for Example \ref{example1} when using $\mathcal{P}^{k}$ polynomials and SDC time
discretization on a uniform mesh of $N$ cells. Final time $t=1$.}
\vspace{0.3cm}
\begin{tabular}{cccccccccc}
\toprule
&N &    $L^{1}$  & order &  $L^{2} $ & order &$L^{\infty} $  & order  \\
\hline
$\mathcal{P}^{1}$&   10&    2.97E-02&     --&    3.61E-02&     --&    9.45E-02&     --\\
&   20&    7.66E-03&     1.96&    9.31E-03&     1.96&    2.39E-02&     1.98\\
&   40&    1.93E-03&     1.99&    2.35E-03&     1.99&    6.04E-03&     1.99\\
&   80&    4.83E-04&     2.00&    5.88E-04&     2.00&    1.51E-03&     2.00\\
&  160&    1.21E-04&     2.00&    1.47E-04&     2.00&    3.79E-04&     2.00\\
&  320&    3.02E-05&     2.00&    3.68E-05&     2.00&    9.46E-05&     2.00\\\hline
$\mathcal{P}^{2}$&   10&    2.63E-02&     --&    2.92E-02&     --&    4.19E-02&     --\\
&   20&    3.57E-03&     2.88&    3.97E-03&     2.88&    5.70E-03&     2.88\\
&   40&    4.54E-04&     2.98&    5.04E-04&     2.98&    7.18E-04&     2.99\\
&   80&    5.68E-05&     3.00&    6.31E-05&     3.00&    8.98E-05&     3.00\\
&  160&    7.10E-06&     3.00&    7.88E-06&     3.00&    1.12E-05&     3.00\\
&  320&    8.87E-07&     3.00&    9.85E-07&     3.00&    1.40E-06&     3.00\\\hline
$\mathcal{P}^{3}$&   10&    1.54E-03&     --&    1.71E-03&     --&    2.44E-03&     --\\
&   20&    1.40E-04&     3.46&    1.55E-04&     3.46&    2.22E-04&     3.46\\
&   40&    9.35E-06&     3.90&    1.04E-05&     3.90&    1.49E-05&     3.90\\
&   80&    5.99E-07&     3.96&    6.66E-07&     3.96&    9.54E-07&     3.96\\
&  160&    3.76E-08&     3.99&    4.18E-08&     3.99&    5.99E-08&     3.99\\
&  320&    2.36E-09&     4.00&    2.62E-09&     4.00&    3.75E-09&     4.00\\
\hline
\end{tabular}
\end{table}

\begin{ex}\label{example2}(Accuracy test for a linear fifth-order problem.)
We consider the following linear fifth-order time-dependent problem.
\begin{align*}
&u_{t}+u_{xxxxx}=0,~~~~~~~(x,t)\in[0,2\pi]\times(0,1],\\
&u(x,0)= \sin(x).
\end{align*}
The exact solution is $$u(x,t)= \sin(x-t).$$
\end{ex}

\begin{table}[!htb]\centering
\caption{\label{table.2}Errors and the corresponding convergence rates
for Example \ref{example2} when using $\mathcal{P}^{k}$ polynomials and SDC
time discretization on a uniform mesh of $N$ cells. Final time $t=1$.}
\vspace{0.3cm}
\begin{tabular}{cccccccccc}
\toprule
&N &    $L^{1}$  & order &  $L^{2} $ & order &$L^{\infty} $  & order  \\
\hline
$\mathcal{P}^{1}$&   10&    8.13E-02&     --&    9.08E-02&     --&    1.44E-01&     --\\
&   20&    2.22E-02&     1.87&    2.47E-02&     1.88&    3.97E-02&     1.86\\
&   40&    5.68E-03&     1.97&    6.32E-03&     1.97&    1.08E-02&     1.88\\
&   80&    1.43E-03&     1.99&    1.59E-03&     1.99&    2.81E-03&     1.94\\
&  160&    3.57E-04&     2.00&    3.98E-04&     2.00&    7.15E-04&     1.98\\
&  320&    8.92E-05&     2.00&    9.95E-05&     2.00&    1.80E-04&     1.99\\\hline
$\mathcal{P}^{2}$&   10&    7.25E-02&     --&    8.07E-02&     --&    1.14E-01&     --\\
&   20&    9.74E-03&     2.90&    1.08E-02&     2.90&    1.53E-02&     2.90\\
&   40&    1.23E-03&     2.98&    1.37E-03&     2.98&    1.94E-03&     2.98\\
&   80&    1.54E-04&     3.00&    1.71E-04&     3.00&    2.42E-04&     3.00\\
&  160&    1.93E-05&     3.00&    2.14E-05&     3.00&    3.03E-05&     3.00\\
&  320&    2.41E-06&     3.00&    2.68E-06&     3.00&    3.79E-06&     3.00\\\hline
$\mathcal{P}^{3}$&   10&    5.44E-03&     --&    6.04E-03&     --&    8.56E-03&     --\\
&   20&    4.13E-04&     3.72&    4.59E-04&     3.72&    6.49E-04&     3.72\\
&   40&    2.60E-05&     3.99&    2.89E-05&     3.99&    4.08E-05&     3.99\\
&   80&    1.64E-06&     3.99&    1.82E-06&     3.99&    2.58E-06&     3.99\\
&  160&    1.02E-07&     4.00&    1.14E-07&     4.00&    1.61E-07&     4.00\\
&  320&    6.41E-09&     4.00&    7.12E-09&     4.00&    1.01E-08&     4.00\\
\hline
\end{tabular}
\end{table}

\begin{ex}\label{example3}(Accuracy test for a nonlinear fourth-order problem.)
We consider the following nonlinear fourth-order time-dependent problem.
\begin{align*}
&u_{t}+(u^{2}u_{xx})_{xx}=f,~~~~~~~x\in[0,2\pi].
\end{align*}
The source term $f$ is chosen so that the exact solution
is $$u(x,t)=e^{-t} \sin(x).$$
\end{ex}
We test this example by the DG scheme
(\ref{linear_forth_scheme1})-(\ref{linear_forth_scheme2}).
Both errors and orders of accuracy are listed in Table \ref{table.3}.
We again observe that our
scheme gives the optimal $(k+1)$-th order of the accuracy for
this nonlinear problem.

\begin{table}[!htb]\centering
\caption{\label{table.3}Errors and the corresponding convergence rates
for Example \ref{example3} when using $\mathcal{P}^{k}$ polynomials on a uniform
mesh of $N$ cells. Final time $t=0.1$.}
\vspace{0.3cm}
\begin{tabular}{cccccccccc}
\toprule
&N &    $L^{1}$  & order &  $L^{2} $ & order &$L^{\infty} $  & order  \\
\hline
$\mathcal{P}^{1}$
&    4&    1.47E-01&       --&    1.93E-01&       --&    3.97E-01&     --\\
&    8&    6.74E-02&     1.12&    8.10E-02&     1.25&    2.28E-01&     0.80\\
&   16&    1.94E-02&     1.80&    2.58E-02&     1.65&    8.21E-02&     1.47\\
&   32&    5.05E-03&     1.94&    6.36E-03&     2.02&    2.45E-02&     1.75\\
&   64&    1.19E-03&     2.08&    1.41E-03&     2.17&    4.33E-03&     2.50\\
\hline
$\mathcal{P}^{2}$
&    4&    4.85E-02&       --&    6.72E-02&       --&    2.63E-01&      --\\
&    8&    2.63E-03&     4.21&    3.77E-03&     4.16&    1.37E-02&     4.26\\
&   16&    8.22E-04&     1.68&    1.38E-03&     1.45&    5.87E-03&     1.23\\
&   32&    1.19E-04&     2.79&    2.12E-04&     2.71&    1.00E-03&     2.55\\
&   64&    1.55E-05&     2.94&    2.68E-05&     2.99&    1.58E-04&     2.67\\
\hline
$\mathcal{P}^{3}$
&    4&    4.86E-03&       --&    5.91E-03&      --&     1.81E-02&     --\\
&    8&    1.07E-03&     2.19&    1.75E-03&     1.75&    8.99E-03&     1.01\\
&   16&    3.54E-05&     4.92&    6.61E-05&     4.73&    4.42E-04&     4.35\\
&   32&    1.16E-06&     4.93&    2.04E-06&     5.02&    1.68E-05&     4.71\\
&   64&    4.65E-08&     4.64&    6.99E-08&     4.87&    5.99E-07&     4.81\\
\hline
\end{tabular}
\end{table}

\begin{ex}\label{example4}(Accuracy test for a nonlinear fifth-order problem.)
We consider the following nonlinear fifth-order time-dependent problem
\begin{align*}
&u_{t}+(u_{xx})^{3}_{xxx}=f,~~~~~~~x\in[0,2\pi],
\end{align*}
where the source term $f$ is chosen such that the exact solution
is $$u(x,t)=\sin(x-t).$$
\end{ex}
We test this example by the DG scheme
(\ref{linear_fifth_scheme1})-(\ref{linear_fifth_scheme3}).  Both
the errors and the numerical orders of accuracy
are listed in Table \ref{table.4}.
We once again observe the designed $(k+1)$-th order of
accuracy for this nonlinear problem.

\begin{table}[!htb]\centering
\caption{\label{table.4}Errors and the corresponding convergence rates
for Example \ref{example4} when using $\mathcal{P}^{k}$ polynomials on a uniform
mesh of $N$ cells. Final time $t=0.1$.}
\vspace{0.3cm}
\begin{tabular}{cccccccccc}
\toprule
&N &    $L^{1}$  & order &  $L^{2} $ & order &$L^{\infty} $  & order  \\
\hline
$\mathcal{P}^{1}$
&    4&    2.06E-01&      -- &    2.33E-01&       --&    5.05E-01&       --\\
&    8&    5.44E-02&     1.92&    6.94E-02&     1.75&    2.09E-01&     1.28\\
&   16&    1.64E-02&     1.73&    2.01E-02&     1.79&    6.13E-02&     1.77\\
&   32&    3.67E-03&     2.16&    4.47E-03&     2.16&    1.42E-02&     2.11\\
&   64&    1.19E-03&     1.62&    1.44E-03&     1.63&    4.17E-03&     1.77\\
\hline
$\mathcal{P}^{2}$
&    4&    3.06E-02&     --&      4.39E-02&     --&      1.72E-01&       --\\
&    8&    4.14E-03&     2.88&    6.34E-03&     2.79&    2.80E-02&     2.62\\
&   16&    4.01E-04&     3.37&    5.56E-04&     3.51&    2.44E-03&     3.52\\
&   32&    4.73E-05&     3.08&    6.78E-05&     3.04&    3.29E-04&     2.89\\
&   64&    5.57E-06&     3.09&    8.34E-06&     3.02&    4.07E-05&     3.02\\
\hline
$\mathcal{P}^{3}$
&    4&    4.91E-03&     --&    6.45E-03&     --&    2.00E-02&    --\\
&    8&    1.42E-04&     5.12&    1.96E-04&     5.04&    1.03E-03&     4.28\\
&   16&    8.95E-06&     3.98&    1.25E-05&     3.98&    6.73E-05&     3.93\\
&   32&    5.06E-07&     4.15&    7.38E-07&     4.08&    4.21E-06&     4.00\\
\hline
\end{tabular}
\end{table}

\bigskip

The last example we consider is a two-dimensional fourth-order problem.
\begin{ex}\label{example5} (Accuracy test for a two-dimensional  linear fourth-order problem.)
We consider the following fourth-order time-dependent problem with
the periodic boundary condition
\begin{align*}
&u_{t}+\Delta^{2}u=0,~~~~~~~(x,y)\in[0,2\pi]\times[0,2\pi],\\
&u(x,0)= \sin(x+y).
\end{align*}
The exact solution is $$u(x,t)=e^{-4t}\sin(x+y).$$  Our computation is
based on the flux choice (\ref{2d_fourth_flux1}). The errors and
numerical orders of accuracy for the $\mathcal{Q}^{k}$ elements
with $1\leq k\leq 3$ are listed in Table \ref{table.5}. We observe
that our scheme gives the optimal $(k+1)$-th order of the accuracy
when $k\geq 1$.
\end{ex}

\begin{table}[!htb]\centering
\caption{\label{table.5}Errors and the corresponding convergence rates
for Example \ref{example5} when using $\mathcal{Q}^{k}$ polynomials on a uniform
mesh of $N\times N$ cells. Final time $t=1$.}
\vspace{0.3cm}
\begin{tabular}{cccccccccc}
\toprule
&$N\times N$&    $L^{1}$  & order &  $L^{2} $ & order &$L^{\infty} $  & order  \\
\hline
$\mathcal{Q}^{1}$
& $4\times4$&      1.67E-01&      --&    2.46E-01&       --&     1.13E+00&     --\\
& $8\times8$&      5.29E-02&     1.66&    7.93E-02&     1.63&    4.04E-01&     1.49\\
& $16\times16$&    1.25E-02&     2.08&    2.03E-02&     1.97&    1.07E-01&     1.92\\
& $32\times32$&    3.02E-03&     2.05&    5.09E-03&     2.00&    2.70E-02&     1.98\\
& $64\times64$&    7.46E-04&     2.02&    1.27E-03&     2.00&    6.78E-03&     2.00\\\hline
$\mathcal{Q}^{2}$&$2\times2$&    3.41E-01&     --&    5.14E-01&     --&    2.55E+00&    --\\
&$4\times4$&    4.49E-02&     2.92&    7.29E-02&     2.82&    5.20E-01&     2.29\\
& $8\times8$   &    5.41E-03&     3.05&    9.03E-03&     3.01&    6.73E-02&     2.95\\
& $16\times16$ &    6.70E-04&     3.01&    1.12E-03&     3.01&    8.45E-03&     2.99\\
&  $32\times32$ &    8.35E-05&     3.00&    1.40E-04&     3.00&    1.06E-03&     3.00\\
&$64\times64$ & 1.04E-05&     3.00&    1.75E-05&     3.00&    1.32E-04&     3.00\\
\hline
\end{tabular}
\end{table}

\section{Concluding remarks}

In this paper, we have constructed a new class of
discontinuous Galerkin methods combining the LDG and UWDG methods for solving high order PDEs, namely time-dependent PDEs
with high order spatial derivatives.  The idea is to
rewrite the PDE into a lower order system, but not to
a system with only first order spatial derivatives as in
LDG methods.  The ideas in designing numerical fluxes to
obtain stable and accurate DG schemes from both the LDG
schemes and the UWDG schemes, including the usage of
alternating and upwinding numerical fluxes when appropriate,
are then used to obtain
stable and optimally convergent DG schemes for a wide
variety of linear and nonlinear PDEs with high order
spatial derivatives in both one and two spatial dimensions.
The main advantage of our method over the LDG method is that we
have introduced fewer auxiliary variables, thereby reducing memory
and computational costs.  The main advantage of our method over
the UWDG method is that no internal penalty terms are necessary
in order to ensure stability for both even and odd order PDEs.
Detailed algorithm formulation, stability analysis and
optimal $L^{2}$ error estimates are given for several examples, including
fourth order linear and nonlinear equations in one dimension and
a fourth order linear equation in two dimension, and
fifth order linear and nonlinear wave equations in one dimension.
In our error estimates, a key ingredient is the study
of the relationship between the derivative and the element interface
jumps of the numerical solution and the auxiliary variable
numerical solution of the derivative.  With this relationship and
by using the discrete Sobolev and Poincar\'{e} inequalities, we
can obtain optimal error estimates for both even order
diffusive PDEs and odd order wave PDEs.  Numerical examples
are provided both for linear and nonlinear equations and both in
one dimension and in two dimensions, to verify the theoretical
results.  Extension of the optimal error estimates to the nonlinear
equations is highly nontrivial and is left for future work.

\end{document}